\documentclass[11pt,a4paper]{amsart}
\usepackage[english]{babel}
\usepackage{graphicx}
\newtheorem{theorem}{Theorem}[section]

\newtheorem{corollary}[theorem]{Corollary}
\newtheorem{lemma}[theorem]{Lemma}
\newtheorem{proposition}[theorem]{Proposition}

\theoremstyle{definition}

\newtheorem{remark}[theorem]{Remark}

\usepackage[square,sort,comma,numbers]{natbib}

\usepackage[all,cmtip,2cell,graph]{xy}
\usepackage{pstricks}
\usepackage{enumerate}
\usepackage{amsfonts,amssymb,amsmath,eucal,pinlabel,array,hhline}
\usepackage{slashed}
\usepackage{tabulary}
\usepackage{fancyhdr}
\usepackage{a4wide}
\usepackage{calrsfs,bbm,dsfont}
\usepackage[position=b]{subcaption}
\usepackage{color}

\UseAllTwocells

\newcommand{\Z}{\mathds{Z}}

\newcommand{\R}{\mathds{R}}

\newcommand{\id}{\mathit{id}}

\newenvironment{romanlist}
	{\begin{enumerate}
	}
	{\end{enumerate}}

\DeclareSymbolFont{EulerScript}{U}{eus}{m}{n}
\DeclareSymbolFontAlphabet\mathscr{EulerScript}
\begin{document}

\title{A Burau-Alexander 2-functor on tangles}

\author{David Cimasoni}
\address{Universit\'e de Gen\`eve, Section de math\'ematiques, 2 rue du Li\`evre, 1211 Gen\`eve, Switzerland}
\email{david.cimasoni@unige.ch}
\author{Anthony Conway }
\address{Universit\'e de Gen\`eve, Section de math\'ematiques, 2 rue du Li\`evre, 1211 Gen\`eve, Switzerland}
\email{anthony.conway@unige.ch}


\begin{abstract}

We construct a weak~$2$-functor from the bicategory of oriented tangles to a bicategory of Lagrangian cospans.
This functor simultaneously extends the Burau representation of the braid groups, its generalization to tangles due to Turaev and the first-named author,
and the Alexander module of~$1$ and~$2$-dimensional links.

\end{abstract}
\maketitle

%%%%%%%%%%%%%%%%%%%%%%%%%%

\section{Introduction}

Since its introduction in 1935, the Burau representation~\cite{Bur} has been one of the most studied representations of the braid groups. In its reduced version,
it takes the form of a homomorphism~$\rho_n\colon B_n\to\mathit{GL}_{n-1}(\Lambda)$ with~$\Lambda=\Z[t^{\pm 1}]$, which preserves some non-degenerate skew-hermitian form
on~$\Lambda^{n-1}$~\cite{Squier}. The construction of~$\rho_n$, whether the algebraic one~\cite{Bur,Bir} or the homological one~\cite{KLW},
easily extend to {\em oriented braids\/}, i.e. braids where different strands can be oriented in different directions. In a slightly pedantic style, one can therefore say that
the Burau representations constitute a functor~$\rho$ from the groupoid~$\mathbf{Braids}$, with objects finite sequences of signs~$\pm 1$ and morphisms oriented braids, to the 
groupoid~$\mathbf{U}_\Lambda$, with objects~$\Lambda$-modules equipped with a non-degenerate skew-hermitian form and morphisms unitary~$\Lambda$-isomorphisms.

In this context, it is natural to ask whether this Burau functor extends to the category~$\mathbf{Tangles}$ of oriented tangles (whose formal definition can be found in subsection~\ref{sub:tangles}). Such an extension was constructed by Turaev and the first-named author in~\cite{CT}. In a nutshell, they defined a category~$\mathbf{Lagr}_\Lambda$ of {\em Lagrangian
relations\/} in which~$\mathbf{U}_\Lambda$ embeds via the graph functor: if~$f$ is a unitary isomorphism, then its graph~$\Gamma_f$ is a Lagrangian relation. Then, they constructed
a functor~$\mathcal{F}\colon\mathbf{Tangles}\to\mathbf{Lagr}_\Lambda$ such that~$\mathcal{F}(\beta)=\Gamma_{\rho(\beta)}$ for any oriented braid~$\beta$.

Note that the groupoid~$\mathbf{Braids}$ is nothing but the {\em core\/} of~$\mathbf{Tangles}$: both categories have the same objects, and the isomorphisms of the latter are
the morphisms of the former. Therefore, one might wonder if there is an extension~$\mathcal{B}$ of~$\rho$ to oriented tangles taking values in a category whose core is (equivalent
to)~$\mathbf{U}_\Lambda$. More importantly, oriented surfaces between oriented tangles turn~$\mathbf{Tangles}$ into a (weak)~$2$-category
(see subsection~\ref{sub:2mor} for a discussion of this fact), and several functors, such as the one coming from Khovanov homology~\cite{Kho}, have been shown to extend to~$2$-functors. Hence, one can also hope to extend~$\mathcal{B}$ to a~$2$-functor. This is what we achieve in the present paper, building on the homological definition of the Burau representation.

The idea is to consider {\em cospans\/}~\cite{Ben,Gra} of~$\Lambda$-modules, i.e. diagrams of the form~$H\to T\leftarrow H'$.
More precisely, we start by defining the category~$\mathbf{L}_\Lambda$ of {\em Lagrangian cospans\/}, whose core is shown to be equivalent
to~$\mathbf{U}_\Lambda$. This category should be understood as a generalization of the category of Lagrangian relations, in the sense that there is a full (non-faithful)
functor~$F\colon\mathbf{L}_\Lambda\to\mathbf{Lagr}_\Lambda$ which is the identity on objects. The functor~$\mathcal{F}$ then lifts
in a very natural way to a functor~$\mathcal{B}$ taking values in~$\mathbf{L}_\Lambda$. In summary, we have the commutative diagram of functors
\[
\xymatrix@R0.5cm{ 
 & \mathbf{L}_\Lambda \ar[d]^F \\
\mathbf{Tangles} \ar[ur]^{\mathcal{B}} \ar[r]_{\mathcal{F}} & \mathbf{Lagr}_\Lambda\,,
}
\]
where~$\mathcal{B}$ extends the Burau functor~$\rho$ in the following sense: the restriction of~$\mathcal{B}$ to~$\mathbf{Braids}=\mathit{core}(\mathbf{Tangles})$ fits in the
commutative diagram
\[
\xymatrix@R0.5cm{ 
 & \mathit{core}(\mathbf{L}_\Lambda) \ar[d]_\simeq ^{F|} \\
\mathbf{Braids} \ar[ur]^{\mathcal{B}|} \ar[r]_{\rho} & \mathbf{U}_\Lambda\,,
}
\]
with the vertical arrow an equivalence of categories.

Furthermore, the category~$\mathbf{L}_\Lambda$ can be modified in a natural way and endowed with a weak~$2$-category structure
yielding a {\em bicategory\/}, and~$\mathcal{B}$ extends to a weak~$2$-functor on the bicategory of oriented tangles and surfaces. Finally, when restricted to~$1$ and~$2$-endomorphisms of the empty set, i.e.
oriented links and closed surfaces,~$\mathcal{B}$ is nothing but the Alexander module.

\medskip

The paper is organized as follows. In Section~\ref{sec:Lagr}, we recall the definition of the category~$\mathbf{Lagr}_\Lambda$ of Lagrangian relations,
we define our category~$\mathbf{L}_\Lambda$ of Lagrangian cospans together with the full functor~$F\colon\mathbf{L}_\Lambda\to\mathbf{Lagr}_\Lambda$, and prove that
the core of~$\mathbf{L}_\Lambda$ is equivalent to~$\mathbf{U}_\Lambda$. In Section~\ref{sec:2cat}, we show that~$\mathbf{L}_\Lambda$ naturally
extends to a bicategory. In Section~\ref{sec:2functor}, we give the definition of the category of oriented tangles, discuss its~$2$-category extension,
construct the functor~$\mathcal{B}$ and show that it extends to a weak~$2$-functor.
Finally, in Section~\ref{sec:versions}, we briefly explain how other versions of the Burau representation
(namely, unreduced and multivariable versions) can be extended to weak~$2$-functors using the same ideas.

%%%%%

\section*{Acknowledgments}
The authors wish to thank Louis-Hadrien Robert for useful discussions, and the anonymous referee for pointing out a mistake in an earlier version of this article.
The first author was supported by the Swiss National Science Foundation. The second author was supported by the NCCR SwissMAP, funded by the Swiss National Science Foundation.

%%%%%%%%%%%%%%%%%%%%%%%%%%

\section{Lagrangian categories}
\label{sec:Lagr}

The aim of this section is to introduce the various algebraic categories that appear in our construction. In a first subsection, we briefly recall the definition of the
category~$\mathbf{Lagr}_\Lambda$ of Lagrangian relations over a ring~$\Lambda$ and explain why it should be understood as a generalization of the groupoid~$\mathbf{U}_\Lambda$ of unitary automorphisms of Hermitian~$\Lambda$-modules, following~\cite{CT}. In subsection~\ref{sub:pushout}, we recall the theory of cospans in a category
with pushouts. In subsection~\ref{sub:cospan}, we define the category~$\mathbf{L}_\Lambda$ of Lagrangian cospans, and relate it to the category of Lagrangian relations via a full functor~$F\colon\mathbf{L}_\Lambda\to\mathbf{Lagr}_\Lambda$. In subsection~\ref{sub:core}, we show that this functor restricts to an equivalence of categories between the core groupoid of~$\mathbf{L}_\Lambda$ and~$\mathbf{U}_\Lambda$.

%%%%%

\subsection{The category of Lagrangian relations}
\label{sub:Lagr}

Fix an integral domain~$\Lambda$ endowed with a ring involution~$\lambda \mapsto \overline{\lambda}$. A {\em skew-Hermitian form\/} on a~$\Lambda$-module~$H$ is a map~$\omega\colon H \times H\rightarrow\Lambda$ such that for
all~$x,y,z \in H$ and all~$\lambda,\lambda'\in\Lambda$,
\begin{romanlist}
\item{$\omega(\lambda x + \lambda' y,z)=\lambda \omega(x,z)+\lambda' \omega(y,z)$,}
\item{$\omega(x,y)=-\overline{\omega(y,x)}$,}
\item{if~$\omega(x,y)=0$ for all~$y \in H$, then~$x=0$.}
\end{romanlist}
A {\em Hermitian~$\Lambda$-module\/}~$H$ is a finitely generated~$\Lambda$-module endowed with a skew-Hermitian form~$\omega$.  The same module~$H$ with the opposite form~$-\omega$ will be denoted by~$-H$. The {\em annihilator\/} of a submodule~$A \subset H$ is the submodule$
\mathrm{Ann}(A)=\lbrace x \in H \ | \ \omega(v,x)=0 \text{ for all } v \in A \rbrace\,.$
A submodule is called {\em Lagrangian\/} if it is equal to its annihilator. Given a submodule~$A$ of a Hermitian~$\Lambda$-module~$H$, set
\[
\overline{A}= \lbrace x \in H \ | \ \lambda x \in A \ \text{for a non-zero} \ \lambda \in \Lambda \rbrace\,.
\]
Observe that if~$A$ is Lagrangian, then~$\overline{A}=A$.

If~$H$ and~$H'$ are Hermitian~$\Lambda$-modules, a {\em Lagrangian relation\/} from~$H$ to~$H'$ is a Lagrangian submodule of~$(-H)\oplus H'$. For instance, given a
Hermitian~$\Lambda$-module~$H$, the {\em diagonal relation\/}
\[
\Delta_H= \lbrace h \oplus h \in H \oplus H \ | \  h \in H \rbrace
\]
is a Lagrangian relation from~$H$ to~$H$. Given two Lagrangian relations~$N_1$ from~$H$ to~$H'$ and~$N_2$ from~$H'$ to~$H''$, their {\em composition\/} is defined
as~$N_2 \circ N_1 := \overline{N_2N_1}\subset (-H) \oplus H''$, where
\[
N_2N_1 = \lbrace x \oplus z \ | \ x \oplus y \in N_1 \ \text{and} \ y \oplus z \in N_2 \text{ for some~$y\in H'$}\rbrace\,.
\]
The proof of the next theorem can be found in~\cite[Theorem 2.7]{CT}.

\begin{theorem}
\label{thm:Lagr}
Hermitian~$\Lambda$-modules, as objects, and Lagrangian relations, as morphisms, form a category. \qed
\end{theorem}

Following~\cite{CT}, we shall denote this category by~$\mathbf{Lagr}_\Lambda$ and call it the {\em category of Lagrangian relations over~$\Lambda$\/}.

We shall say that a~$\Lambda$-linear map between two Hermitian~$\Lambda$-modules is {\em unitary\/} if it preserves the corresponding skew-Hermitian forms.
Let us now briefly recall why Lagrangian relations can be understood as a generalization of unitary~$\Lambda$-isomorphisms and unitary~$Q$-isomorphisms, where~$Q=Q(\Lambda)$ is the field of fractions of~$\Lambda$. Let~$\mathbf{U}_\Lambda$ be the category of Hermitian~$\Lambda$-modules and unitary~$\Lambda$-isomorphisms. Also, let~$\mathbf{U}_\Lambda^0$ be the category of Hermitian~$\Lambda$-modules, where the morphisms between~$H$ and~$H'$ are the unitary~$Q$-isomorphisms between~$H\otimes Q$ and~$H' \otimes Q$.
The {\em graph\/} of a~$\Lambda$-linear map~$f \colon H \rightarrow H'$ is the submodule~$\Gamma_f=\lbrace x \oplus f(x) \rbrace$ of~$H \oplus H'$.  Similarly the
{\em restricted graph\/} of a~$Q$-linear map~$\varphi \colon H \otimes Q \rightarrow H' \otimes Q$ is~$\Gamma_\varphi^0=\Gamma_\varphi \cap (H \oplus H').$ The proof of the following theorem can be found in \cite[Theorem~$2.9$]{CT}.

\begin{theorem}
\label{thm:graph}
The maps~$f \mapsto f\otimes \mathit{id}_Q$,~$f \mapsto \Gamma_f$ and~$\varphi \mapsto \Gamma^0_\varphi$ define faithful functors which are the identity on objects, and
fit in the commutative diagram
\[
\xymatrix{\mathbf{U}_\Lambda \ar[r]^{- \otimes Q} \ar@/_1pc/[rr]_\Gamma & \mathbf{U}_\Lambda^0 \ar[r]^{\Gamma^0}& \mathbf{Lagr}_\Lambda\,.}
\]\qed
\end{theorem}

We shall call such functors {\em embeddings of categories\/}.

%%%%%

\subsection{Cospans in a category with pushouts}
\label{sub:pushout}

Among the arguments that will be used in this article, some are well-known and of purely categorical nature. This subsection contains a quick review of these results
(see~\cite{Ben,Reb} for further detail).

Let us fix a category~$\mathbf{C}$. Throughout this subsection, all objects, morphisms, diagrams, and the like will be in this fixed category~$\mathbf{C}$.
Recall that a {\em span\/} is a diagram of the form~$T_1\stackrel{i_1}{\longleftarrow}H\stackrel{i_2}{\longrightarrow}T_2$. A {\em pushout\/} of such a span
is an object~$P$ together with morphisms~$T_1\stackrel{j_1}{\longrightarrow}P\stackrel{j_2}{\longleftarrow}T_2$ such that~$j_1 i_1=j_2 i_2$, which satisfies the following universal property:
for any~$T_1\stackrel{k_1}{\longrightarrow}Q\stackrel{k_2}{\longleftarrow}T_2$ such that~$k_1 i_1=k_2 i_2$, there exists a unique morphism~$u\colon P\to Q$
with~$u j_1=k_1$ and~$u j_2=k_2$. This is illustrated in the following commutative diagram:
\[
\xymatrix@R0.5cm{
 & Q  & & \\
 & P \ar@{.>}[u]^{u}  & & \\
T_1 \ar[ru]^{j_1} \ar@/^1pc/[uur]^{k_1}  & & T_2. \ar[ul]_{j_2}\ar@/_1pc/[uul]_{k_2} \\
  & H\ar[lu]_{i_1}\ar[ru]^{i_2} & }
\]
If a span admits a pushout, then the latter is unique up to canonical isomorphism. However, not all spans admit pushouts in general.
From now on, we shall assume that~$\mathbf{C}$ is a {\em category with pushouts\/}, i.e. that any span admits a pushout. Moreover,
we fix for each span a pushout. 

Let~$H,H'$ be two objects. A {\em cospan\/} from~$H$ to~$H'$ is a diagram~$H\stackrel{i}{\longrightarrow}T\stackrel{i'}{\longleftarrow}H'$.
Two cospans~$H\stackrel{i_1}{\longrightarrow}T_1\stackrel{i_1'}{\longleftarrow}H'$ and~$H\stackrel{i_2}{\longrightarrow}T_2\stackrel{i_2'}{\longleftarrow}H'$ are
{\em isomorphic\/} if there is an isomorphism~$f\colon T_1\rightarrow T_2$ such that~$f i_1=i _2$ and~$f i_1' = i_2'$.
The {\em composition\/} of two cospans~$H\stackrel{i_1}{\longrightarrow}T_1\stackrel{i_1'}{\longleftarrow}H'$
and~$H'\stackrel{i'_2}{\longrightarrow}T_2\stackrel{i_2''}{\longleftarrow}H''$ is the cospan from~$H$ to~$H''$ given by the (fixed) pushout diagram
\[
\xymatrix@R0.5cm{
& & T_2 \circ T_1 & & & \\
& T_1 \ar[ru]^{j_1} & & T_2 \ar[ul]_{j_2} \\
H \ar[ru]^{i_1} & & H' \ar[ru]^{i'_2}\ar[lu]_{i_1'} & &H''\,. \ar[lu]_{i_2''}&}
\]
Finally, the {\em identity cospan\/} of an object~$H$ is defined as the cospan~$I_H:=(H\stackrel{\mathit{id}}{\longrightarrow}H\stackrel{\mathit{id}}{\longleftarrow}H)$.

\begin{remark}
\label{rem:cospan}
Given any morphism~$H'\stackrel{i'}{\longrightarrow}H''$, one easily checks that
\[
\xymatrix@R0.5cm{
 & H'' & \\
H'\ar[ru]^{i'} &  & \quad H'' \ar[lu]_{\id}\\
  & H'\ar[lu]_{\id}\ar[ru]^{i'} & } \ \ \
\xymatrix@R0.5cm{
 & H'' & \\
H''\ar[ru]^{\id} &  & H' \ar[lu]_{i'}\\
  & H'\ar[lu]_{i'}\ar[ru]^{\id} & }
\]
are pushout diagrams. Therefore, if one makes this choice of pushout for spans of the
form~$H'\stackrel{\id}{\longleftarrow}H'\stackrel{i'}{\longrightarrow}H''$ and
~$H''\stackrel{i'}{\longleftarrow}H'\stackrel{\id}{\longrightarrow}H'$, then
the composition of~$H\stackrel{i}{\longrightarrow}H'\stackrel{\id}{\longleftarrow}H'$ and~$H'\stackrel{i'}{\longrightarrow}T\stackrel{i''}{\longleftarrow}H''$ is given
by~$H\stackrel{i'i}{\longrightarrow}T\stackrel{i''}{\longleftarrow}H''$, and the composition of~$H\stackrel{i}{\longrightarrow}T\stackrel{i'}{\longleftarrow}H'$
and~$H'\stackrel{\id}{\longrightarrow}H'\stackrel{i''}{\longleftarrow}H''$ is given by~$H\stackrel{i}{\longrightarrow}T\stackrel{i'i''}{\longleftarrow}H''$.
For this reason, cospans should be understood as generalizing morphisms in the category~$\mathbf{C}$.
\end{remark}

Note that the composition of cospans depends on the choice of a pushout for each span; therefore, it cannot be associative for all such choices.
For the same reason, the composition does not admit~$I_H$ as a two-sided unit in general. However, for any fixed choice of pushouts, the corresponding composition
does satisfy these properties up to canonical isomorphisms of cospans. We refer the reader to~\cite{Reb} for a proof of this standard fact in the dual context
of spans.

There are two possible strategies at this point. The first one, which we will use in the remaining part of Section~\ref{sec:Lagr}, is to consider the category
given by the objects of~$\mathbf{C}$, as objects, and the isomorphism classes of cospans in~$\mathbf{C}$, as morphisms. The second one, which we will use in the next sections,
is to follow the ``main principle of category theory " as stated in~\cite[p.179]{K-V}, that is: not to identify isomorphic cospans, but to view these canonical isomorphisms as part of the (higher) structure. This naturally leads to the concept of a {\em bicategory\/}, that will be reviewed in subsection~\ref{sub:2-cat} and used
in subsections~\ref{sub:2-cospan} and~\ref{sub:2mor}.

%%%%%

\subsection{The category~$\mathbf{L}_\Lambda$ of Lagrangian cospans} 
\label{sub:cospan}

We now take~$\mathbf{C}$ to be the category of~$\Lambda$-modules, with~$\Lambda$ any integral domain. After observing that this is a category with pushouts, we impose
further conditions on our cospans and work with isomorphism classes thereof.

We begin with the following standard result, whose easy proof is left to the reader. 

\begin{lemma}
\label{lemma:pushout}
The square
\[
\xymatrix@R0.5cm{
& P & \\
T_1 \ar[ru]^{j_1} & & T_2 \ar[ul]_{j_2} \\
& H \ar[ru]_{i_2}\ar[lu]^{i_1} &}
\]
is a pushout diagram in the category of~$\Lambda$-modules if and only if the sequence 
\[
\xymatrix@R0.5cm{
H\ar[r]^{\!\!\!\!\!\!\!(-i_1,i_2)} & T_1\oplus T_2\ar[r]^{\;\;\;\;\;{j_1\choose j_2}} & P \ar[r]& 0\,
}
\]
is exact.\qed
\end{lemma}

In particular, a pushout is given by the cokernel of the map~$(-i_1,i_2) \colon H \rightarrow T_1 \oplus T_2$ sending~$x$ to~$(-i_1(x)) \oplus i_2(x)$, so this
is a category with pushouts.

By abuse of notation, we shall sometimes simply denote by~$T$ (the isomorphism class of) a cospan of the form~$H\to T\leftarrow H'$.
For a cospan~$H\stackrel{i}{\longrightarrow}T\stackrel{i'}{\longleftarrow}H'$, consider the submodule~$N_T:=\mathit{Ker}{-i\choose \phantom{-}i'}$ of~$H \oplus H'$,
where~${-i\choose \phantom{-}i'}\colon H \oplus H' \rightarrow T$ maps~$(x,y)$ to~$i'(y)-i(x)$.
Note that if~$T_1$ and~$T_2$ are isomorphic cospans, then~$N_{T_1}$ and~$N_{T_2}$ are equal. 

\begin{lemma}
\label{lem:composition}
For any two composable cospans~$T_1$ and~$T_2$, we have~$N_{T_2 \circ T_1}=N_{T_2} N_{T_1}$.
\end{lemma}

\begin{proof}
Consider two cospans~$H\stackrel{i_1}{\longrightarrow}T_1\stackrel{i_1'}{\longleftarrow}H'$ and~$H'\stackrel{i'_2}{\longrightarrow}T_2\stackrel{i_2''}{\longleftarrow}H''$.
By definition,~$N_{T_2 \circ T_1}$ is the kernel of the map~$H\oplus H''\to T_2\circ T_1$ given by~$(x,z)\mapsto j_2(i_2''(z))-j_1(i_1(x))$. Since~$T_2 \circ T_1$ is
represented by the cokernel of the map~$(-i_1',i'_2)\colon H' \rightarrow T_1 \oplus T_2$,~$N_{T_2 \circ T_1}$ consists of the elements~$x \oplus z\in H\oplus H''$ for
which~$(-i_1(x))\oplus i_2''(z)$ lies in the image of~$(-i_1',i'_2)$. Therefore,~$N_{T_2 \circ T_1}$ is equal to
\[
\lbrace x \oplus z \in H \oplus H'' |  \ i_1(x)=i_1'(y) \text{ and } i'_2(y)=i_2''(z) \text{ for some } y \in H'   \rbrace\,.
\]
In other words,~$N_{T_2 \circ T_1}$ is equal to~$\mathit{Ker}{-i_2'\choose \phantom{-}i_2''}\mathit{Ker}{-i_1\choose \phantom{-}i_1'}=N_{T_2}N_{T_1}$.
\end{proof}

Recall from the subsection~\ref{sub:Lagr} that if~$A$ is a submodule of a Hermitian~$\Lambda$-module~$H$, then~$ \overline{A}$ consists of all~$x \in H$ such that~$\lambda x\in A$
for a non-zero~$\lambda\in\Lambda$. We shall say that a cospan~$H\stackrel{i}{\longrightarrow}T\stackrel{i'}{\longleftarrow}H'$ is {\em Lagrangian\/}
if~$\overline{N_T}$ is a Lagrangian submodule of~$(-H)\oplus H'$. For instance, the identity cospan~$I_H$ is a Lagrangian cospan, since~$\overline{N_H}=N_H$ is equal to the diagonal relation~$\Delta_H$.
 
\begin{proposition}
\label{prop:1Lagr}
Hermitian~$\Lambda$-modules, as objects, and isomorphism classes of Lagrangian cospans, as morphisms, form a category~$\mathbf{L}_\Lambda$.
Moreover, the map~$T\mapsto\overline{N_T}$ gives rise to a full functor~$F\colon \mathbf{L}_\Lambda\rightarrow\mathbf{Lagr}_\Lambda$.
\end{proposition}

\begin{proof}
As explained in subsection~\ref{sub:pushout}, it is a standard fact that the composition of isomorphism classes of cospans by pushouts is well-defined (i.e. does not depend on the choice of the pushouts), is associative,
with the identity cospan acting trivially. Therefore, we only need to check that the composition of two Lagrangian cospans~$H\to T_1\leftarrow H'$ and~$H'\to T_2\leftarrow H''$
is also Lagrangian, and that~$F$ is a full functor. By~Lemma~\ref{lem:composition}, we have
\[
\overline{N_{T_2\circ T_1}}=\overline{N_{T_2}N_{T_1}}=N_{T_2}\circ N_{T_1}\,.
\]
Since~$N_{T_2}$ are~$N_{T_1}$ are Lagrangian,~$N_{T_2}=\overline{N_{T_2}}, N_{T_1}=\overline{N_{T_1}}$ and~$N_{T_2}\circ N_{T_1}$ is also Lagrangian by Theorem~\ref{thm:Lagr}.
Therefore, the cospan~$H\to T_2 \circ T_1\leftarrow H''$ is Lagrangian and~$F$ is a functor. Finally, given a Lagrangian relation~$N$ from~$H$ to~$H'$, consider the
cospan~$H\stackrel{i}{\longrightarrow}T\stackrel{i'}{\longleftarrow}H'$ where~$T=(H\oplus H')/N$ and~$i$ (resp.~$i'$) is the inclusion of~$H$ (resp.~$H'$) into~$H\oplus H'$ composed
with the canonical projection. By construction, it is a Lagrangian cospan with~$\overline{N_T}=\overline{N}=N$, so the functor~$F$ is full.
\end{proof}

Let us conclude this subsection by noting that the functor~$F\colon \mathbf{L}_\Lambda\rightarrow\mathbf{Lagr}_\Lambda$ is {\em not\/} faithful.
Indeed, given any cospan~$H\stackrel{i}{\longrightarrow}T\stackrel{i'}{\longleftarrow}H'$ and any~$\Lambda$-module~$\widetilde{T}$, consider the cospan given
by~$H\stackrel{(i,0)}{\longrightarrow}T\oplus\widetilde{T}\stackrel{(i',0)}{\longleftarrow}H'$. One immediately checks the equality~$N_{T\oplus\widetilde{T}}=N_T$.
Therefore, if the first cospan is Lagrangian and~$\widetilde{T}$ is non-trivial, then these two cospans represent different morphisms in~$\mathbf{L}_\Lambda$ mapped by~$F$
to the same morphism in~$\mathbf{Lagr}_\Lambda$.

%%%%%

\subsection{The core of the category~$\mathbf{L}_\Lambda$}
\label{sub:core}

Recall that the {\em core\/} of a category~$\mathcal{C}$ is the maximal sub-groupoid of~$\mathcal{C}$. In other words,~$\mathit{core}(\mathcal{C})$ is
the subcategory of~$\mathcal{C}$ consisting of all objects of~$\mathcal{C}$ and with morphisms all the isomorphisms of~$\mathcal{C}$.

We shall say that a cospan~$H\stackrel{i}{\longrightarrow}T\stackrel{i'}{\longleftarrow}H'$ is~{\em invertible} if~$i$ and~$i'$ are~$\Lambda$-isomorphisms.

\begin{proposition}
\label{prop:core}
The core of~$\mathbf{L}_\Lambda$ consists of Hermitian~$\Lambda$-modules, as objects, and isomorphism classes of invertible Lagrangian
cospans, as morphisms. Furthermore, the map assigning to such a cospan~$H\stackrel{i}{\longrightarrow}T\stackrel{i'}{\longleftarrow}H'$
the~$\Lambda$-isomorphism~$i'^{-1}i\colon H\to H'$ gives rise to an equivalence of categories~$\mathit{core}(\mathbf{L}_\Lambda) \stackrel{\simeq}{\longrightarrow}\mathbf{U}_\Lambda$
which fits in the commutative diagram
\[
\xymatrix@R0.5cm{ 
\mathit{core}(\mathbf{L}_\Lambda)\ar[r]\ar[d]_\simeq&\mathbf{L}_\Lambda\ar[d]^F\\
\mathbf{U}_\Lambda\ar[r]^{\Gamma}&\mathbf{Lagr}_\Lambda\,.
}
\]
\end{proposition}

\begin{proof}
Given a cospan~$H\stackrel{i}{\longrightarrow}T\stackrel{i'}{\longleftarrow}H'$ with~$i$ an isomorphism, one easily checks that the diagram
\[
\xymatrix@R0.5cm{
& & H & & \\
& T \ar[ru]^{i^{-1}} & & T \ar[ul]_{i^{-1}} & \\
H \ar[ru]^{i} & & H' \ar[ru]^{i'}\ar[lu]_{i'} & &H \ar[lu]_{i}}
\]
satisfies the universal property for the pushout defining the composition of~$H\stackrel{i}{\longrightarrow}T\stackrel{i'}{\longleftarrow}H'$ with~$H'\stackrel{i'}{\longrightarrow}T\stackrel{i}{\longleftarrow}H$. Hence, if both~$i$ and~$i'$ are isomorphisms, then these cospans are inverse of one another, and therefore isomorphisms in~$\mathbf{L}_\Lambda$, i.e. morphisms in~$\mathit{core}(\mathbf{L}_\Lambda)$.
Conversely, let~$H\stackrel{i_1}{\longrightarrow}T_1\stackrel{i_1'}{\longleftarrow}H'$ be a morphism in~$\mathit{core}(\mathbf{L}_\Lambda)$, and
let~$H'\stackrel{i_2'}{\longrightarrow}T_2\stackrel{i_2}{\longleftarrow}H$ be its inverse. Working with the cokernel representatives of~$T_2\circ T_1$ and~$T_1\circ T_2$, this means
that there exist~$\Lambda$-isomorphisms~$C:=\mathit{Coker}(-i_1',i'_2)\stackrel{\varphi}{\to} H$ and~$C':=\mathit{Coker}(-i_2,i_1)\stackrel{\varphi'}{\to} H'$ such that the following diagrams commute:
\[
\xymatrix@R0.5cm{
& & H & & \\
& & C \ar[u]^{\varphi}_{\simeq} & & \\
& T_1 \ar[ru]^{j_1} & & T_2 \ar[ul]_{j_2} &\\
H \ar[ru]^{i_1} \ar@/^2pc/[rruuu]^{\mathit{id}_H} & & H' \ar[ru]^{i'_2}\ar[lu]_{i_1'} & &H \ar[lu]_{i_2}\ar@/_2pc/[lluuu]_{\mathit{id}_H}}
 \ %\qquad
\xymatrix@R0.5cm{
& & H' & & \\
& & C' \ar[u]^{\varphi'}_{\simeq} & & \\
& T_2 \ar[ru]^{j'_2} & & T_1 \ar[ul]_{j'_1} &\\
H' \ar[ru]^{i'_2} \ar@/^2pc/[rruuu]^{\mathit{id}_{H'}} & & H \ar[ru]^{i_1}\ar[lu]_{i_2} & &H' \ar[lu]_{i'_1}\ar@/_2pc/[lluuu]_{\mathit{id}_{H'}}\,.}
\]
This implies that the following diagram has exact rows, and is commutative:
\[
\xymatrix{
0 \ar[r] & H\ar[r]^{\!\!\!\!\!\!\!(-1,1)}\ar[d]_{\varphi'j_1' i_1}& H\oplus H \ar[r]^{\;\;\;\;\;{1\choose 1}}\ar[d]^{i_1\oplus i_2}&H\ar[d]^{\varphi^{-1}}_\simeq\ar[r]& 0\phantom{\,.}\\
0 \ar[r] & H'\ar[r]_{\!\!\!\!\!\!\!(-i_1',i'_2)} & T_1\oplus T_2\ar[r]_{\;\;\;\;\;{j_1\choose j_2}} & C \ar[r]& 0\,.
}
\]
Using the universal property of the pushouts~$T_2\circ T_1$ and~$T_1\circ T_2$, one can check that the maps~$\varphi' j_1'i_1=\varphi' j_2'i_2\colon H\to H'$
and~$\varphi j_1i_1'=\varphi j_2i'_2\colon H'\to H$ are inverse of each other, and therefore isomorphisms. By the five-lemma applied to the diagram above,~$i_1$
and~$i_2$ are also isomorphisms. Exchanging the roles of~$T_1$ and~$T_2$ leads to the same conclusion for~$i_1'$ and~$i'_2$, so both these cospans are invertible.

Now, let~$G\colon\mathit{core}(\mathbf{L}_\Lambda) \rightarrow \mathbf{U}_\Lambda$ be defined by assigning to the invertible
cospan~$H\stackrel{i}{\longrightarrow}T\stackrel{i'}{\longleftarrow}H'$
the~$\Lambda$-isomorphism~$i'^{-1}i\colon H\to H'$. First note that isomorphic cospans are mapped to the same isomorphism. Next, observe that for any two invertible 
cospans~$H\stackrel{i_1}{\longrightarrow}T_1\stackrel{i_1'}{\longleftarrow}H'$ and~$H'\stackrel{i'_2}{\longrightarrow}T_2\stackrel{i_2}{\longleftarrow}H''$, we have
\[
G(T_2 \circ T_1)=(j_2i_2)^{-1}(j_1i_1)=i_2^{-1}j_2^{-1}j_1i_1=i_2^{-1}i'_2i_1'^{-1}i_1 = G(T_2) \circ G(T_1)\,.
\]
(Here, we used the fact that since~$j_2i_2$ and~$i_2$ are isomorphisms, so is~$j_2$.)
We now check that if a cospan~$H\stackrel{i}{\longrightarrow}T\stackrel{i'}{\longleftarrow}H'$ is Lagrangian, then~$i'^{-1}i$ is unitary. Indeed,
\[
F(T)=\overline{N_T}=\overline{\mathit{Ker}{\textstyle{-i\choose \phantom{-}i'}}}=\mathit{Ker}{\textstyle{-i\choose \phantom{-}i'}}=\mathit{Ker}{\textstyle{-i'^{-1}i\choose\mathit{id}}}=\Gamma_{i'^{-1}i}
\]
is a Lagrangian subspace of~$(-H)\oplus H'$. Therefore, for any~$x\in H$ and~$y\in H'$, we have
\[
0=(-\omega \oplus \omega')(x \oplus i'^{-1}(i(x)), y \oplus i'^{-1}(i(y)))=-\omega(x,y)+\omega'(i'^{-1}(i(x)),i'^{-1}(i(y)))\,,
\]
so~$i'^{-1}i$ is indeed unitary. The equality~$F(T)=\Gamma_{i'^{-1}i}$ displayed above also shows the commutativity of the diagram in the statement.

It only remains to check that~$G$ is a fully-faithful functor. Given any unitary isomorphism~$f\colon H\to H'$, the
cospan~$H\stackrel{f}{\longrightarrow}H'\stackrel{\mathit{id}}{\longleftarrow}H'$ is invertible, Lagrangian, and is mapped to~$f$ by~$G$. Finally,
if~$H\stackrel{i_1}{\longrightarrow}T_1\stackrel{i_1'}{\longleftarrow}H'$ and~$H\stackrel{i_2}{\longrightarrow}T_2\stackrel{i_2'}{\longleftarrow}H'$ are invertible cospans
with~$i_1'^{-1}i_1=i_2'^{-1}i_2$, then the map~$i_2i_1^{-1}=i'_2i_1'^{-1}\colon T_1\to T_2$ defines an isomorphism between these two cospans.
\end{proof}

Since~$\Lambda$ is an integral domain, we can consider its quotient field~$Q=Q(\Lambda)$. We shall call a cospan~$H\stackrel{i}{\longrightarrow}T\stackrel{i'}{\longleftarrow}H'$
{\em rationally invertible\/} if~$i_Q:= i \otimes Q$ and~$i_Q':=i' \otimes Q$ are~$Q$-isomorphisms. The next proposition can be checked by the same arguments as
Proposition~\ref{prop:core}. We therefore leave the proof to the reader.

\begin{proposition}
\label{prop:rationally}
Hermitian~$\Lambda$-modules, as objects, and isomorphism classes of rationally invertible Lagrangian cospans, as morphisms, form a
category~$\mathit{core}(\mathbf{L}_\Lambda)^0$. Furthermore, the map assigning to such a cospan~$H\stackrel{i}{\longrightarrow}T\stackrel{i'}{\longleftarrow}H'$
the~$Q$-isomorphism~$i_Q'^{-1}i_Q$ gives rise to a full functor~$\mathit{core}(\mathbf{L}_\Lambda)^0\rightarrow\mathbf{U}^0_\Lambda$ which fits in
the following commutative diagram:
\[
\xymatrix@R0.5cm{ 
\mathit{core}(\mathbf{L}_\Lambda)^0\ar[r]\ar[d]&\mathbf{L}_\Lambda\ar[d]^F\\
\mathbf{U}^0_\Lambda\ar[r]^{\Gamma^0}&\mathbf{Lagr}_\Lambda\,.
}
\]
\qed
\end{proposition}

Summarizing this section, we have six categories which all have Hermitian~$\Lambda$-modules as objects. They fit in the following commutative diagram
\begin{equation}
\label{equ:summary}
\xymatrix@R0.5cm{ 
\mathit{core}(\mathbf{L}_\Lambda) \ar[r] \ar[d]_\simeq
 & \mathit{core}(\mathbf{L}_\Lambda)^0  \ar[d]\ar[r] & \mathbf{L}_\Lambda \ar[d]^F\\
\mathbf{U}_\Lambda \ar[r]^{-\otimes Q} \ar@/_2pc/[rr]_\Gamma & \mathbf{U}_\Lambda^0 \ar[r]^{\Gamma^0}& \mathbf{Lagr}_\Lambda\,,}
\end{equation}
where the horizontal arrows are embeddings of categories, the left-most vertical arrow is an equivalence of categories, and the two remaining ones are full functors.

%%%%%%%%%%%%%%%%%%%%%%%%%%%%%%%%%

\section{A bicategory of Lagrangian cospans}
\label{sec:2cat}

The aim of this section is to endow~$\mathbf{L}_\Lambda$ with the structure of a bicategory. We begin by recalling in
subsection~\ref{sub:2-cat} the notions of bicategory and weak~$2$-functor, before defining the bicategory
of Lagrangian cospans in subsection~\ref{sub:2-cospan}.

%%%%%

\subsection{2-categories and 2-functors}
\label{sub:2-cat}

Following the original work of B\'enabou~\cite{Ben}, it is a traditional practice to use the term ``$2$-category'' for what Kapranov and Voevodsky call
a ``strict~$2$-category''~\cite{K-V}. As it turns out, the categories that appear in our work are not of this type, but have a richer structure: that of some type of weak~$2$-category
known as a {\em bicategory\/}.  We now recall the definition of this structure, following~\cite{Ben}.

A {\em bicategory\/}~$\mathcal{C}$ consists of the following data:
\begin{enumerate}[(i)]
\item{A set~$\mathrm{Ob}\mathcal{C}$ whose elements are called {\em objects\/}.}
\item{For each pair of objects~$(X,Y)$, a category~$\mathcal{C}(X,Y)$ whose objects are called {\em~$1$-morphisms} and denoted by~$f\colon X\to Y$ or by~$X\stackrel{f}{\to}Y$,
whose morphisms are called {\em~$2$-morphisms} and denoted by~$\alpha\colon f\Rightarrow g$, or by~$\xymatrix@C+2pc{X\rtwocell^{f}_{g}{\;\;\;\alpha}&Y}​$, and whose composition
is called {\em vertical composition\/} and denoted by
\[
\left(\xymatrix@C+2pc{X\rtwocell^{f}_{g}{\;\;\;\alpha}&Y},\xymatrix@C+2pc{X\rtwocell^{g}_{h}{\;\;\;\beta}&Y}\right)\mapsto\xymatrix@C+2pc{X\rtwocell^{f}_{h}{\;\;\;\;\;\;\beta\star\alpha}&Y}\,.
\]
We shall denote the identity morphism for~$f$ by~$\mathit{Id}_f\colon f\Rightarrow f$.
}
\item{For each object~$X$, an {\em identity\/}~$1$-morphism~$I_X\colon X\to X$.}
\item{For each triple of objects~$(X,Y,Z)$, a functor~$\mathcal{C}(X,Y)\times\mathcal{C}(Y,Z)\to\mathcal{C}(X,Z)$ denoted by
\[
\left(\xymatrix@C+2pc{X\rtwocell^{f}_{g}{\;\;\;\alpha}&Y},\xymatrix@C+2pc{Y\rtwocell^{k}_{\ell}{\;\;\;\beta}&Z}\right)\mapsto\xymatrix@C+2pc{X\rtwocell^{k\circ f}_{\ell\circ g}{\;\;\;\;\;\;\beta\bullet\alpha}&Z\,,}
\]
and called the {\em horizontal composition\/} functor.}
\end{enumerate}
Note that the functoriality of this composition boils down to the identity 
\[
\mathit{Id}_f \bullet \mathit{Id}_g = \mathit{Id}_{g\circ f}
\]
for any composable~$1$-morphisms~$f$ and~$g$, and to the {\em interchange law\/}
\[
(\delta \bullet \beta) \star (\gamma \bullet \alpha)=(\delta \star \gamma) \bullet (\beta \star \alpha)\,,
\]
for each composable~$2$-morphisms~$\alpha,\beta,\gamma$ and~$\delta$. This last condition is best understood by saying that the following~$2$-morphism is well-defined, i.e. independent of the order of the compositions:
\[
\xymatrix@C+2pc{
X\ruppertwocell^{}{\;\;\;\alpha} \ar[r] \rlowertwocell_{}{\;\;\;\beta}  & Y \ruppertwocell^{}{\;\;\;\gamma} \ar[r] \rlowertwocell_{}{\;\;\;\delta} & Z\,.} 
\]

If this horizontal composition is associative (both on the~$1$-morphisms and~$2$-morphisms) and admits~$I_X$ as a two-sided unit, then we are in the presence
of a (strict)~$2$-category. As mentioned above, a bicategory has a richer structure: the horizontal composition is associative and unital
only up to natural isomorphisms, which are part of the structure. To be more precise, a bicategory also contains the following data:
\begin{enumerate}[(i)]\setcounter{enumi}{4}
\item{For any triple of composable~$1$-morphisms~$f,g,h$, an invertible~$2$-morphism
\[
a=a_{fgh}\colon (h\circ g)\circ f\Rightarrow h\circ(g\circ f)
\]
which is natural in~$f,g$ and~$h$, and called the {\em associativity isomorphism\/}.}
\item{For any~$1$-morphism~$X\stackrel{f}{\to}Y$, two invertible~$2$-morphisms~$\ell=\ell_f\colon I_Y\circ f\Rightarrow f$ and~$r=r_f\colon f\circ I_X \Rightarrow f$ which are natural in~$f$.}
\end{enumerate}

These natural isomorphisms must satisfy the following two coherence axioms.
Given four composable~$1$-morphisms~$e,f,g,h$, there are two natural ways to pass from~$((h\circ g)\circ f)\circ e$ to~$h\circ(g\circ(f\circ e))$ using the associativity isomorphisms, one in two steps, the other one in three.
The {\em associativity coherence axiom \/} requires that these two compositions coincide.
Finally, given any~$1$-morphisms~$X\stackrel{f}{\to}Y\stackrel{g}{\to}Z$, the {\em identity coherence axiom\/} requires the
composition~$(g\circ I_Y)\circ f\stackrel{a}{\Longrightarrow} g\circ(I_Y\circ f)\stackrel{\mathit{Id}_g\bullet\ell_f}{\Longrightarrow}g\circ f$ to coincide
with~$r_g\bullet\mathit{Id}_f$. 

\bigskip

Let us now recall the definition of a weak~$2$-functor, also known as a {\em pseudofunctor\/}~\cite{Gro} and originally called a {\em homomorphism of bicategories\/} by
B\'enabou~\cite{Ben}.

Given two bicategories~$\mathcal{C}$ and~$\mathcal{D}$, a {\em weak~$2$-functor\/}~$\mathcal{F}\colon\mathcal{C}\to\mathcal{D}$
consists of the following data:
\begin{enumerate}[(i)]
\item{A map~$\mathcal{F}\colon\mathrm{Ob}\mathcal{C}\to\mathrm{Ob}\mathcal{D}$.}
\item{For each pair of objects~$(X,Y)$ in~$\mathcal{C}$, a functor~$\mathcal{C}(X,Y)\to\mathcal{D}(\mathcal{F}(X),\mathcal{F}(Y))$
denoted by 
\[
\xymatrix@C+2pc{X\rtwocell^{f}_{g}{\;\;\;\alpha}&Y}\mapsto\xymatrix@C+2pc{\mathcal{F}(X)\rtwocell^{\mathcal{F}(f)}_{\mathcal{F}(g)}{\;\;\;\;\;\;\mathcal{F}(\alpha)}&\mathcal{F}(Y)}\,.
\]}
\end{enumerate}
Note that this functoriality is equivalent to the identities
\[
\mathcal{F}(\beta\star \alpha)=\mathcal{F}(\beta)\star\mathcal{F}(\alpha)\quad\text{and}\quad\mathcal{F}(\mathit{Id}_f)=\mathit{Id}_{\mathcal{F}(f)}\,.
\]
If we also have the identities~$\mathcal{F}(I_X)=I_{\mathcal{F}(X)}$,~$\mathcal{F}(g\circ f)=\mathcal{F}(g)\circ\mathcal{F}(f)$
and~$\mathcal{F}(\beta\bullet\alpha)=\mathcal{F}(\beta)\bullet\mathcal{F}(\alpha)$, then we are in the presence of a (strict)~$2$-functor.
Our functor has a finer structure: once again, these identities hold up to isomorphisms of functors, which are part of the data as follows.
\begin{enumerate}[(i)]\setcounter{enumi}{2}
\item{For each object~$X$ of~$\mathcal{C}$, an invertible~$2$-morphism~$\varphi_X\colon I_{\mathcal{F}(X)}\Rightarrow\mathcal{F}(I_X)$ in~$\mathcal{D}$.}
\item{For each~$X\stackrel{f}{\to}Y\stackrel{g}{\to}Z$, an invertible~$2$-morphism~$\varphi=\varphi_{fg}\colon\mathcal{F}(g)\circ\mathcal{F}(f)\Rightarrow\mathcal{F}(g\circ f)$
in~$\mathcal{D}$ such that for any
\[
\xymatrix@C+2pc{X\rtwocell^{f}_{f'}{\;\;\;\alpha}&Y\rtwocell^{g}_{g'}{\;\;\;\beta}&Y\,,}
\]
we have the following equality of~$2$-morphisms in~$\mathcal{D}$:
\[
\varphi_{f'g'}\star(\mathcal{F}(\beta)\bullet\mathcal{F}(\alpha))=\mathcal{F}(\beta\bullet\alpha)\star\varphi_{fg}\,\colon\,\mathcal{F}(g)\circ\mathcal{F}(f)\Rightarrow\mathcal{F}(g'\circ f')\,.
\]}
\end{enumerate}
Finally, these isomorphisms of functors are required to satisfy two coherence axioms.
\begin{enumerate}[(1)]
\item{For any composable~$1$-morphisms~$f,g,h$ of~$\mathcal{C}$, the composition
\[
(\mathcal{F}(h)\circ\mathcal{F}(g))\circ\mathcal{F}(f)\stackrel{a}{\Rightarrow}\mathcal{F}(h)\circ(\mathcal{F}(g)\circ\mathcal{F}(f))
\stackrel{\mathit{Id}\bullet\varphi}{\Longrightarrow}\mathcal{F}(h)\circ\mathcal{F}(g\circ f)\stackrel{\varphi}{\Rightarrow}
\mathcal{F}(h\circ(g\circ f))
\]
is equal to the composition
\[
(\mathcal{F}(h)\circ\mathcal{F}(g))\circ\mathcal{F}(f)\stackrel{\varphi\bullet\mathit{Id}}{\Longrightarrow}\mathcal{F}(h\circ g)\circ\mathcal{F}(f)\stackrel{\varphi}{\Rightarrow}
\mathcal{F}((h\circ g)\circ f)\stackrel{\mathcal{F}(a)}{\Longrightarrow}\mathcal{F}(h\circ(g\circ f))\,.
\]}
\item{For any~$X\stackrel{f}{\to}Y$, the composition
\[
I_{\mathcal{F}(Y)}\circ \mathcal{F}(f)\stackrel{\varphi_Y\bullet\mathit{Id}}{\Longrightarrow}\mathcal{F}(I_Y)\circ \mathcal{F}(f)
\stackrel{\varphi}{\Rightarrow}\mathcal{F}(I_Y\circ f)\stackrel{\mathcal{F}(\ell)}{\Longrightarrow}\mathcal{F}(f)
\]
coincides with~$\ell_{\mathcal{F}(f)}$, and similarly for~$r$.}
\end{enumerate}

%%%%%

\subsection{A bicategory of cospans}
\label{sub:2-cospan}

Our goal is now to define a bicategory of Lagrangian cospans. To do so, we will first work in the more general setting of a category with pushouts.
We wish to emphasize that our resulting bicategory of cospans differs from the usual definition considered in the literature, where the~$2$-morphisms are usually taken to be morphisms of cospans (see e.g.~\cite{Ben}). On the other hand, a notion dual to the~$2$-morphisms that we consider was already studied by Morton~\cite{Morton} in another context (see also \cite{Reb}).

Throughout this section,~$\mathbf{C}$ is a category with pushouts in which we fix a pushout for each span.

The objects of our bicategory are the objects of~$\mathbf{C}$ and the~$1$-morphisms are the cospans in~$\mathbf{C}$, where the horizontal composition is given by our choice of a fixed pushout. It remains to define the~$2$-morphisms, the vertical composition, the associativity and identity isomorphisms, and what is left of the horizontal composition.
 A {\em~$2$-cospan in~$\mathbf{C}$} from~$H\stackrel{i_1}{\longrightarrow}T_1\stackrel{i_1'}{\longleftarrow}H'$ to~$H\stackrel{i_2}{\longrightarrow}T_2\stackrel{i_2'}{\longleftarrow}H'$ consists of a
cospan~$T_1 \stackrel{\alpha_1}{\longrightarrow}A\stackrel{\alpha_2}{\longleftarrow}T_2$ in~$\mathbf{C}$
for which the two following squares commute
\[
\xymatrix@R0.5cm{ 
& T_1\ar[d]^{\alpha_1}  & \\
H \ar[ur]^{i_1} \ar[dr]_{i_2} & A & \;H'\,.\ar[ul]_{i'_1} \ar[dl]^{i'_2} \\
& T_2 \ar[u]_{\alpha_2} & 
}
\]
Two such~$2$-cospans~$T_1 \stackrel{\alpha_1}{\longrightarrow}A\stackrel{\alpha_2}{\longleftarrow}T_2$ and~$T_1 \stackrel{\alpha_1'}{\longrightarrow}A'\stackrel{\alpha_2'}{\longleftarrow}T_2$ are said to be {\em isomorphic\/} if there is a~$\mathbf{C}$-isomorphism~$f \colon A \to A'$ such that~$f\alpha_1=\alpha_1'$
and~$f\alpha_2=\alpha_2'$. Abusing notation, we shall often denote the isomorphism class of such a~$2$-cospan by~$A\colon T_1\Rightarrow T_2$.
These will be the~$2$-morphisms in our~$2$-category.

Let us now proceed with the definition of the {\em vertical composition\/} of the~$2$-morphisms~$A\colon T_1\Rightarrow T_2$ and~$B\colon T_2\Rightarrow T_3$. It is best explained by the diagram
\[
\xymatrix@R0.5cm{
 & T_1\ar[d]_{\alpha_1}  & \\
 & A & \\
H \ar@/^1pc/[uur]^{i_1} \ar[r]^{i_2} & T_2 \ar[u]^{\alpha_2} & H'\ar[l]_{i_2'}\ar@/_1pc/[uul]_{i'_1}\\
 & \star & \\
H \ar@/_1pc/[ddr]_{i_3} \ar[r]^{i_2} & T_2 \ar[d]_{\beta_2} & H'\ar[l]_{i_2'} \ar@/^1pc/[ddl]^{i'_3} \\
 & B & \\
 & T_3\ar[u]^{\beta_3} & \\
} \ \
\xymatrix@R0.5cm{\\ \\ \\ \\ = \\ \\ \\}
\ \ 
\xymatrix@R0.5cm{\\ \\ 
& T_1\ar[d]^{v_A\alpha_1}  & \\
H \ar@/^1pc/[ur]^{i_1} \ar@/_1pc/[dr]_{i_3} & B \star A & H',\ar@/_1pc/[ul]_{i_1'} \ar@/^1pc/[dl]^{i_3'} \\
& T_3 \ar[u]_{v_B\beta_3} & \\
\\
\\
}
\]
where~$B \star A$ and~$v_A, v_B$ are given by the pushout diagram
\[
\xymatrix@R0.5cm{
& & B \star A & & & \\
& A \ar[ru]^{v_A} & & B \ar[ul]_{v_B} \\
T_1\ar[ru]^{\alpha_1} & & T_2 \ar[ru]^{\beta_2}\ar[lu]_{\alpha_2} & &T_3. \ar[lu]_{\beta_3}&}
\]
One can easily check that this indeed defines a~$2$-cospan.

\begin{remark}
\label{rem:vert}
In the special case where~$\alpha_2=\id_{T_2}$ (resp.~$\beta_2=\id_{T_2}$), this vertical
composition~$T_1\stackrel{v_A\alpha_1}{\longrightarrow}B\star A\stackrel{v_B\beta_3}{\longleftarrow}T_3$
is isomorphic to~$T_1\stackrel{\beta_2\alpha_1}{\longrightarrow}B\stackrel{\beta_3}{\longleftarrow}T_3$ (resp.~$T_1\stackrel{\alpha_1}{\longrightarrow}A\stackrel{\alpha_2\beta_3}{\longleftarrow}T_3$).
This is a direct consequence of Remark~\ref{rem:cospan}.
\end{remark}

On the level~$2$-morphisms, the {\em horizontal composition\/} of~$A\colon T_1\Rightarrow T_2$ and~$B\colon T_3\Rightarrow T_4$ is described by the diagram
\[
\xymatrix@R0.5cm{ & T_3 \ar[d]^{\beta_3} & \\
H' \ar[ur]^{i_3'} \ar[dr]_{i_4'} & B & H''\ar[ul]_{i''_3} \ar[dl]^{i''_4}  \\
& T_4 \ar[u]_{\beta_4} }
 \ 
\xymatrix@R0.5cm{ \\ \bullet \\ }
 \ 
\xymatrix@R0.5cm{ 
& T_1\ar[d]^{\alpha_1} & \\
H \ar[ur]^{i_1} \ar[dr]_{i_2} & A & H'\ar[ul]_{i'_1} \ar[dl]^{i'_2} \\
& T_2 \ar[u]_{\alpha_2} & 
}
 \
\xymatrix@R0.5cm{ \\ = \\ }
 \ 
\xymatrix@R0.5cm{ 
& T_3\circ T_1\ar[d]^{h_{31}}  & \\
H \ar@/^1pc/[ur]^{j_1i_1} \ar@/_1pc/[dr]_{j_2i_2} & B \bullet A & \;H''\,,\ar@/_1pc/[ul]_{j_3i_3''} \ar@/^1pc/[dl]^{j_4i_4''} \\
& T_4\circ T_2 \ar[u]_{h_{42}} & 
}
\]
where~$j_1,\dots,j_4$ are the maps that arise in the compositions~$T_3\circ T_1$ and~$T_4\circ T_2$ (see the diagrams below),~$B \bullet A$ is given by the pushout
\[
\xymatrix@R0.5cm{
& & B \bullet A & & & \\
& A \ar[ru]^{h_A} & & B \ar[ul]_{h_B} \\
H\ar[ru]^{\alpha_1i_1} & & H' \ar[ru]^{\beta_3i_3'}\ar[lu]_{\alpha_2i_2'} & &H''\,,\ar[lu]_{\beta_4i_4''}&}
\]
and the maps~$h_{31}$ and~$h_{42}$ are obtained as follows. Since~$h_A \alpha_1 i_1'=h_A \alpha_2 i_2'=h_B \beta_3 i_3'$
and~$h_A\alpha_2 i_2'=h_B\beta_3i_3'=h_B\beta_4i_4'$, the pushout diagrams
\begin{equation}
\label{equ:LegsForHori}
\xymatrix@R0.4cm{
& & B \bullet A & & & \\
& & T_3\circ T_1 \ar[u]^{h_{31}} & & & \\
& A \ar@/^1pc/[uur]^{h_A} &  & B \ar@/_1pc/[uul]_{h_B} &  & \\
& T_1 \ar[ruu]^{j_1} \ar[u]^{\alpha_1} & & T_3 \ar[uul]_{j_3} \ar[u]_{\beta_3} \\
 & & H' \ar[lu]_{i_1'}\ar[ru]^{i_3'} & &}
\xymatrix@R0.4cm{
& & B \bullet A & & & \\
& & T_4\circ T_2 \ar[u]^{h_{42}} & & & \\
& A \ar@/^1pc/[uur]^{h_A} &  & B \ar@/_1pc/[uul]_{h_B} &  & \\
& T_2 \ar[ruu]^{j_2} \ar[u]^{\alpha_2} & & T_4 \ar[uul]_{j_4} \ar[u]_{\beta_4} \\
 & & H' \ar[lu]_{i_2'}\ar[ru]^{i_4'} & &}
\end{equation}
provide maps~$h_{31}$ and~$h_{42}$ which turn~$T_3 \circ T_1 \stackrel{h_{31}}{\longrightarrow}B \bullet A\stackrel{h_{42}}{\longleftarrow}T_4 \circ T_2$ into
a~$2$-cospan, as one easily checks. 

The proof of the following theorem  can be found in \cite{Reb} in the dual context of spans. It applies without change to the present setting.

\begin{theorem}
\label{thm:2cat}
Let~$\mathbf{C}$ be a category with pushouts in which a choice of pushout is fixed for each span.
Objects in~$\mathbf{C}$, as objects, cospans in~$\mathbf{C}$, as morphisms, and isomorphism classes of 2-cospans in~$\mathbf{C}$, as 2-morphisms, form a bicategory.
\end{theorem}

Note that strictly speaking, this bicategory depends on the choice of pushouts. However,
another choice would give a bicategory isomorphic in an obvious sense (see e.g.~\cite[p.22]{Ben}).

The special case where~$\mathbf{C}$ is the category of~$\Lambda$-modules and the morphisms are Lagrangian cospans yields the following corollary.

\begin{corollary}
\label{cor:2cat}
Fix a pushout for each span of~$\Lambda$-modules. Hermitian~$\Lambda$-modules, as objects, Lagrangian cospans, as morphisms, and isomorphism classes of~$2$-cospans,
as~$2$-morphisms, form a bicategory.\qed
\end{corollary}

We shall call it ``the'' {\em bicategory of Lagrangian cospans\/}.

%%%%%%%%%%%%%%%%%%%%%%%%%%%%%%%

\section{The Burau-Alexander 2-functor}
\label{sec:2functor}

The aim of this section is to define a weak~$2$-functor~$\mathcal{B}$ from the bicategory of oriented tangles to the bicategory of Lagrangian cospans
where~$\Lambda$ is the ring~$\Z[t^{\pm 1}]$ of Laurent polynomials in one variable with integer coefficients. We proceed in two steps:  in Subsection~\ref{sub:tangles}, we recall the definition of the category of oriented tangles, and construct a functor~$\mathcal{B} \colon \textbf{Tangles} \to \textbf{L}_\Lambda$.
In Subsections~\ref{sub:2tangles} and~\ref{sub:2mor}, we study the bicategory of tangles, and convert~$\mathcal{B}$ into a weak~$2$-functor with values in the bicategory of Lagrangian cospans.

%%%%%

\subsection{The functor~$\mathcal{B}$ on objects and~$1$-morphisms}
\label{sub:tangles}

We start by recalling the definition of the category of oriented tangles.
Let~$D^2$ be the closed unit disk in~$\R^2$. Given a non-negative integer~$n$, let~$x_j$ be the point~$((2j-n-1)/n,0)$
in~$D^2$, for~$j=1,\dots,n$. Let~$\varepsilon$ and~$\varepsilon'$ be sequences of~$\pm 1$'s of respective length~$n$ and~$n'$.
An~{\em$(\varepsilon,\varepsilon')$-tangle\/} is a pair consisting of the cylinder~$D^2\times [0,1]$ and an oriented smooth~$1$-submanifold~$\tau$ whose oriented boundary
is~$\sum_{j=1}^{n'}\varepsilon'_j(x'_j,1)-\sum_{j=1}^{n}\varepsilon_i(x_j,0)$. Note that a~$(\emptyset,\emptyset)$-tangle is nothing but an oriented link.
Two~$(\varepsilon,\varepsilon')$-tangles~$\tau_1$ and~$\tau_2$ are {\em isotopic\/} if there exists an isotopy~$h_t$ of~$D^2\times [0,1]$, keeping~$D^2\times\{0,1\}$ fixed, such
that~$h_1\vert_{\tau_1}\colon\tau_1\simeq\tau_2$ is an orientation-preserving homeomorphism. We shall denote by~$I_\varepsilon$ the isotopy class of the
trivial~$(\varepsilon,\varepsilon)$-tangle~$(D^2,\{x_1,\dots,x_n\})\times [0,1]$. 

Given an~$(\varepsilon,\varepsilon')$-tangle~$\tau_1$ and an~$(\varepsilon',\varepsilon'')$-tangle~$\tau_2$, their {\em composition\/} is
the~$(\varepsilon,\varepsilon'')$-tangle~$\tau_2\circ\tau_1$ obtained by gluing the two cylinders along the disk corresponding to~$\varepsilon'$, smoothing it if needed,
and shrinking the length of the resulting cylinder by a factor~$2$ (see Figure \ref{fig:composition}).
Clearly, the composition induces a composition on the isotopy classes of tangles, which is associative and admits~$I_\varepsilon$ as a~$2$-sided unit. Therefore,
the sequences of~$\pm 1$'s, as objects, and the isotopy classes of tangles, as morphisms, form a category denoted by~$\mathbf{Tangles}$ and called the
{\em category of oriented tangles\/}.

\begin{figure}[tb]
\labellist\small\hair 2.5pt
\pinlabel {$\varepsilon$} at 17 763
\pinlabel {$\varepsilon'$} at 118 767
\pinlabel {$\varepsilon'$} at 160 767
\pinlabel {$\varepsilon''$} at 264 767
\pinlabel {$\varepsilon$} at 360 763
\pinlabel {$\varepsilon''$} at 464 767
\pinlabel {$\tau_1$} at 65 815
\pinlabel {$\tau_2$} at 210 815
\pinlabel {$\tau_2\circ\tau_1$} at 410 815
\endlabellist
\centering
\includegraphics[width=0.6\textwidth]{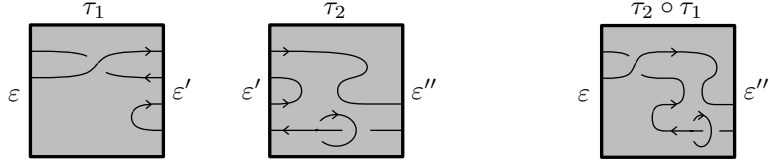}
\caption{An~$(\varepsilon,\varepsilon')$-tangle~$\tau_1$ with~$\varepsilon=(+1,-1)$ and~$\varepsilon'=(-1,+1,-1,+1)$, an~$(\varepsilon',\varepsilon'')$-tangle~$\tau_2$
with~$\varepsilon''=(-1,+1)$, and their composition, the~$(\varepsilon,\varepsilon'')$-tangle~$\tau_2\circ\tau_1$.}
\label{fig:composition}
\end{figure}

Recall that a tangle~$\tau\subset D^2 \times [0,1]$ is called an {\em oriented braid\/} if every component of~$\tau$ is strictly increasing or strictly decreasing with respect to the projection onto~$[0,1]$. The finite sequences of~$\pm 1$'s as objects, and the isotopy classes of oriented braids, as morphisms, form a
subcategory~$\mathbf{Braids}$ of~$\mathbf{Tangles}$, which is nothing but its core. Finally a tangle~$\tau\subset D^2\times [0,1]$ is called an {\em oriented string link\/}
if every component of~$\tau$ joins~$D^2 \times\lbrace 0 \rbrace$ and~$D^2 \times \lbrace 1 \rbrace$. Isotopy classes of oriented string links are the morphisms of a
category~$\mathbf{Strings}$ which satisfies
\[
\mathbf{Braids} \subset \mathbf{Strings} \subset \mathbf{Tangles}\,,
\]
where all the inclusions denote embeddings of categories.

\medbreak

We are now ready to define our Burau-Alexander functor~$\mathcal{B}\colon\mathbf{Tangles}\to\mathbf{L}_\Lambda$. We start by defining it on objects, following the construction
of~\cite{CT}. Denote by~$\mathcal{N}(\lbrace x_1,\dots,x_n\rbrace)$ an open tubular neighborhood of~$\lbrace x_1,\dots, x_n \rbrace$ in~$D^2\subset\R^2$, and let~$S^2$ be the~$2$-sphere obtained by the one-point compactification of~$\R^2$. Given a sequence~$\varepsilon=(\varepsilon_1,\dots,\varepsilon_n)$ of~$\pm 1$,
set~$\ell_\varepsilon=\sum_{i=1}^n \varepsilon_i$ and endow the compact surface
\[
D_\varepsilon = \left\{\begin{array}{lr}
        D^2\setminus\mathcal{N}(\lbrace x_1,\dots,x_n\rbrace) & \text{if } \ell_\varepsilon\neq 0\\
        S^2\setminus\mathcal{N}(\lbrace x_1,\dots,x_n\rbrace) & \text{if } \ell_\varepsilon=0
        \end{array}\right.
\]
with an orientation (pictured counterclockwise), a base point~$z$, and the generating family~$\lbrace e_1,\dots, e_n \rbrace$ of~$\pi_1(D_\varepsilon,z)$, where~$e_i$ is a simple loop turning once around~$x_i$ counterclockwise if~$\varepsilon_i=+1$, clockwise if~$\varepsilon_i=-1$. The same space with the opposite orientation will be denoted by~$-D_\varepsilon$.

The natural epimomorphism~$H_1(D_\varepsilon)\to\mathbb{Z}$, given by~$e_j \mapsto 1$
induces an infinite cyclic covering~$\widehat{D}_\varepsilon \rightarrow D_\varepsilon$ whose homology is endowed with a structure of module over~$\Lambda=\Z[t^{\pm 1}]$.
If~$\ell_\varepsilon\neq 0$, then~$D_\varepsilon$ obviously retracts by deformation on the wedge of~$n$ circles representing the generators~$e_1,\dots,e_n$ of~$\pi_1(D_\varepsilon,z)$, and one can check that~$H_1(\widehat{D}_\varepsilon)$ is a free~$\Lambda$-module of rank~$n-1$. (It is free of rank~$n-2$ if~$\ell_\varepsilon$ vanishes.)
If~$\langle\ , \ \rangle\colon H_1(\widehat{D}_\varepsilon)\times H_1(\widehat{D}_\varepsilon)\rightarrow\Z$ denotes the skew-symmetric intersection form obtained by lifting the orientation of~$D_\varepsilon$
to~$\widehat{D}_\varepsilon$, then the formula
\[
\omega_\varepsilon(x,y)=\sum_{k \in \mathbb{Z}} \langle t^kx,y \rangle t^{-k}
\]
defines a skew-Hermitian~$\Lambda$-valued pairing on~$H_1(\widehat{D}_\varepsilon)$ which is non-degenerate by~\cite[Lemma 3.2]{CT}.
(This is the only reason for considering~$S^2$ instead of~$D^2$ when~$\ell_\varepsilon$ vanishes.) Therefore, following the terminology of
subsection~\ref{sub:Lagr},~$\mathcal{B}(\varepsilon):=(H_1(\widehat{D}_\varepsilon),\omega_\varepsilon)$ is a free Hermitian~$\Lambda$-module for any object~$\varepsilon$ of the category of oriented tangles. Note that this coincides with the definition of the Lagrangian functor~$\mathcal{F}\colon\mathbf{Tangles}\to\mathbf{Lagr}_\Lambda$ of~\cite{CT}
at the level of objects.

Let us now turn to morphisms. First note that the existence of an~$(\varepsilon,\varepsilon')$-tangle~$\tau\subset D^2 \times [0,1]$ implies
that~$\ell_\varepsilon=\ell_{\varepsilon'}$. Denote by~$\mathcal{N}(\tau)$ an open tubular neighborhood of~$\tau$ in~$D^2 \times [0,1]$. We shall orient the exterior
\[
X_\tau = \left\{\begin{array}{lr}
        (D^2 \times [0,1]) \setminus \mathcal{N}(\tau) & \text{if } \ell_\varepsilon\neq 0\\
        (S^2 \times [0,1]) \setminus \mathcal{N}(\tau) & \text{if } \ell_\varepsilon=0
        \end{array}\right.
\]
of~$\tau$ so that the induced orientation on~$\partial X_\tau$ extends the orientation on the space~$(-D_\varepsilon) \sqcup D_{\varepsilon'}$. Clearly, the abelian group~$H_1(X_\tau)$ is generated by the oriented meridians of the connected components of~$\tau$. The homomorphism~$H_1(X_\tau)\rightarrow \mathbb{Z}$ mapping these meridians to~$1$ extends the previously defined homomorphisms~$H_1(D_\varepsilon) \rightarrow \mathbb{Z}$
and~$H_1(D_{\varepsilon'})\rightarrow \mathbb{Z}$. It determines an infinite cyclic covering~$\widehat{X}_\tau \rightarrow X_\tau$ whose homology is  endowed with a
structure of module over~$\Lambda$.

Let~$i_\tau\colon H_1(\widehat{D}_\varepsilon) \rightarrow H_1(\widehat{X}_\tau)$ and~$i_\tau'\colon H_1(\widehat{D}_{\varepsilon'}) \rightarrow H_1(\widehat{X}_\tau)$ be the homomorphisms induced by the inclusions of~$\widehat{D}_\varepsilon$ and~$\widehat{D}_{\varepsilon'}$ into~$\widehat{X}_\tau$.
Since~$\mathcal{F}(\tau)=\overline{\mathit{Ker}{-i_\tau\choose\phantom{-}i_\tau'}}$ is a Lagrangian submodule of~$(-H_1(\widehat{D}_\varepsilon))\oplus H_1(\widehat{D}_\varepsilon')$~\cite[Lemma~$3.3$]{CT}, it follows that~$H_1(\widehat{D}_\varepsilon) \stackrel{i_\tau}\longrightarrow H_1(\widehat{X}_\tau) \stackrel{i_\tau'}\longleftarrow H_1(\widehat{D}_{\varepsilon'})$ is a Lagrangian cospan for any~$1$-morphism~$\tau$ in the category of oriented tangles. Note that the equality above is nothing but the definition of the Lagrangian
functor~$\mathcal{F}$ of~\cite{CT} at the level of morphisms.

\begin{theorem}
\label{thm:functor}
For any sequence~$\varepsilon$ of~$\pm 1$'s, set~$\mathcal{B}(\varepsilon)=(H_1(\widehat{D}_\varepsilon),\omega_\varepsilon)$ and for any isotopy class~$\tau$ of tangles,
let~$\mathcal{B}(\tau)$ denote the isomorphism class of the Lagrangian cospan~$H_1(\widehat{D}_\varepsilon) \stackrel{i_\tau}\longrightarrow H_1(\widehat{X}_\tau) \stackrel{i_\tau'}\longleftarrow H_1(\widehat{D}_{\varepsilon'})$. This defines a functor~$\mathcal{B}\colon\mathbf{Tangles}\to\mathbf{L}_\Lambda$ which fits in the commutative diagram 
\[
\xymatrix@R0.5cm{ 
\mathbf{Braids} \ar[d] \ar[r] \ar@/_3pc/[dd]_\rho & \mathbf{String} \ar[d] \ar[r] & \mathbf{Tangles} \ar[d]^{\mathcal{B}} \ar@/^2.5pc/[dd]^{\mathcal{F}}\ \\
\mathit{core}(\mathbf{L}_\Lambda) \ar[d]_\simeq\ar[r] & \mathit{core}(\mathbf{L}_\Lambda)^0 \ar[d]\ar[r] & \mathbf{L}_\Lambda \ar[d]^F \\ 
\mathbf{U}_\Lambda \ar[r]^{- \otimes Q} \ar@/_1.5pc/[rr]_\Gamma & \mathbf{U}^0_\Lambda \ar[r]^{\Gamma^0} & \mathbf{Lagr}_\Lambda,}
\]
where the left-most vertical arrow is the Burau functor, the horizontal arrows are the embeddings of categories described in Subsections~\ref{sub:Lagr}
and~\ref{sub:tangles}, and~$F$ is the full functor defined in Subsection~\ref{sub:cospan} (recall diagram~(\ref{equ:summary})).
Furthermore, if~$\tau$ is an oriented link, then~$\mathcal{B}(\tau)$ is nothing but its Alexander module.
\end{theorem}

\begin{proof}
For any object~$\varepsilon$, the cospan associated to the identity tangle~$I_\varepsilon$ is canonically isomorphic to the identity
cospan~$I_{\mathcal{B}(\varepsilon)}$. Let us now check that
given~$\tau_1 \in T(\varepsilon,\varepsilon')$ and~$\tau_2 \in T(\varepsilon',\varepsilon'')$, we have the
equality~$\mathcal{B}(\tau_2\circ\tau_1)=\mathcal{B}(\tau_2)\circ\mathcal{B}(\tau_1)$. Let~$H_1(\widehat{D}_\varepsilon) \stackrel{i_1}\rightarrow H_1(\widehat{X}_{\tau_1}) \stackrel{i_1'}\leftarrow H_1(\widehat{D}_{\varepsilon'})$ and~$H_1(\widehat{D}_{\varepsilon'}) \stackrel{i_2'}\rightarrow H_1(\widehat{X}_{\tau_2}) \stackrel{i_2''}\leftarrow H_1(\widehat{D}_{\varepsilon''})$ be the Lagrangian cospans arising from~$\tau_1$ and~$\tau_2$. We must show that~$H_1(\widehat{D}_\varepsilon) \stackrel{k_1i_1}{\longrightarrow} H_1(\widehat{X}_{\tau_3 \circ \tau_1}) \stackrel{k_2i_2''}\longleftarrow H_1(\widehat{D}_{\varepsilon''})$ is isomorphic to the composition~$H_1(\widehat{D}_\varepsilon) \stackrel{j_1i_1}{\longrightarrow} H_1(\widehat{X}_{\tau_3})\circ H_1(\widehat{X}_{\tau_1}) \stackrel{j_2i_2''}\longleftarrow H_1(\widehat{D}_{\varepsilon''})$, where~$k_1,k_2$ are the inclusion induced maps and~$j_1,j_2$ are maps resulting from any representative of the pushout~$H_1(\widehat{X}_{\tau_3})\circ H_1(\widehat{X}_{\tau_1})$.
Observe that~$\widehat{X}_{\tau_2\circ\tau_1}$ decomposes as the union of~$\widehat{X}_{\tau_1}$ and~$\widehat{X}_{\tau_2}$ glued along~$\widehat{D}_{\varepsilon'}$. Therefore, the associated Mayer-Vietoris exact sequence
\[
\xymatrix@R0.5cm{
H_1(\widehat{D}_{\varepsilon'})\ar[r]^{\!\!\!\!\!\!\!\!\!\!\!\!\!\!\!\!\!\!\!(-i_1',i'_2)}&H_1(\widehat{X}_{\tau_1})\oplus H_1(\widehat{X}_{\tau_2})\ar[r]^{\;\;\;\;\;\;\;\;\;\;{k_1\choose k_2}}&H_1(\widehat{X}_{\tau_2\circ\tau_1})\ar[r]& 0
}
\]
together with Lemma~\ref{lemma:pushout} imply that~$H_1(\widehat{X}_{\tau_1}) \stackrel{k_1}{\longrightarrow} H_1(\widehat{X}_{\tau_2 \circ \tau_1}) \stackrel{k_2}\longleftarrow H_1(\widehat{X}_{\tau_2})$ is a representative of the pushout~$H_1(\widehat{X}_{\tau_1}) \circ H_1(\widehat{X}_{\tau_2})$. The claim follows.

Then, observe that the Lagrangian functor~$\mathcal{F}$ is by definition the composition of the functors~$\mathcal{B}\colon\mathbf{Tangles}\to\mathbf{L}_\Lambda$
and~$F\colon\mathbf{L}_\Lambda\to\mathbf{Lagr}_\Lambda$. Also, if~$\tau$ is an oriented string link, then~$\mathcal{B}(\tau)$ is a rationally invertible cospan
by~\cite[Lemma~2.1]{KLW}, and thus belongs to~$\mathit{core}(\mathbf{L}_\Lambda)^0$ by definition. If~$\tau$ is an oriented braid on the other hand, then~$\mathcal{B}(\tau)$
is obviously an invertible cospan, and therefore belongs to~$\mathit{core}(\mathbf{L}_\Lambda)$ by Proposition~\ref{prop:core}.
Finally, if~$\tau$ is a~$(\emptyset,\emptyset)$-tangle, that is, an oriented link~$L$, then the associated Lagrangian
cospan is given by~$0\to H_1(\widehat X_L)\leftarrow 0$, with~$X_L$ the complement of~$L$ in the~$3$-ball. A straightforward Mayer-Vietoris argument shows that considering~$L$
in the~$3$-ball or in the~$3$-sphere does not change the Alexander module, and the proof is completed.
\end{proof}

%%%%%

\subsection{The bicategory of tangles}
\label{sub:2tangles}

The aim is now to convert~$\mathcal{B}$ to a weak~$2$-functor. To do so, we first need to understand how tangles form a (possibly weak)~$2$-category. Once this is done, we will switch from the category~$\textbf{L}_\Lambda$ to the bicategory of Lagrangian cospans and define the
weak~$2$-functor in subsection~\ref{sub:2mor}.

One might think that tangles produce a~$2$-category in a straightforward way~\cite{Fischer}: simply define the objects and~$1$-morphisms as in~$\mathbf{Tangles}$, and the~$2$-morphisms as isotopy classes of oriented surfaces in~$D^2\times [0,1]\times [0,1]$. However, the corresponding vertical composition is not well-defined: indeed, one needs to paste two surfaces along
isotopic tangles, and since the space of tangles isotopic to a fixed one is not necessarily simply-connected, different choices of isotopies can lead to different surfaces.

There are a couple of ways to circumvent this difficulty. One of them is to restrict the space of tangles whose isotopy classes form the~$1$-morphisms, so that the corresponding space of isotopic tangles has trivial fundamental group. Such a construction was given by Kharlamov and Turaev in~\cite{KT} (see also~\cite{BL1}): they considered the class of so-called
{\em generic tangles\/}, and proved that the space of generic tangles isotopic to a fixed one (through generic tangles) is simply-connected, thus obtaining a strict~$2$-category. However, it is more natural in our setting
to take the following alternative approach: define~$1$-morphisms as oriented tangles, and consider isotopies bewteen tangles as part of the ``higher structure''.

\begin{figure}[tb]
\labellist\small\hair 2.5pt
\pinlabel {$\varepsilon$} at 45 530
\pinlabel {$\varepsilon'$} at 556 530
\pinlabel {$\tau_1$} at 303 568
\pinlabel {$\tau_2$} at 350 50
\pinlabel {$\Sigma$} at 330 340
\endlabellist
\centering
\includegraphics[width=0.25\textwidth]{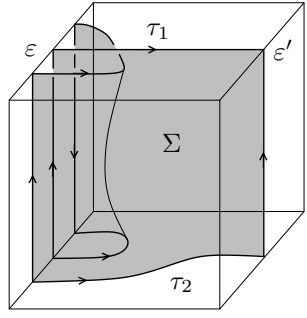}
\caption{A cobordism~$\Sigma\subset D^2\times[0,1]\times[0,1]$ between two~$(\varepsilon,\varepsilon')$-tangles~$\tau_1$ and~$\tau_2$, with~$\varepsilon=(+1,+1,-1)$
and~$\varepsilon'=(+1)$.}
\label{fig:cobordism}
\end{figure}

Let us be more precise. The objects of this bicategory are sequences~$\varepsilon$ of~$\pm 1$'s, while the~$1$-morphisms from~$\varepsilon$ to~$\varepsilon'$ are
the~$(\varepsilon,\varepsilon')$-tangles in~$D^2\times[0,1]$ that are trivial near the top and bottom of the cylinder. (This is to ensure that the composition of two tangles remains a smooth~$1$-submanifold.)

Given two~$(\varepsilon,\varepsilon')$-tangles~$\tau_1$ and~$\tau_2$, a {\em~$(\tau_1,\tau_2)$-cobordism\/} is a pair consisting of
the~$4$-ball~$D^2\times [0,1]\times[0,1]$ together with a proper oriented smooth~$2$-submanifold~$\Sigma$ whose oriented boundary is given by
\[
\partial\Sigma=(\tau_2 \times \{0\})\cup(\varepsilon'\times\{1\}\times[0,1])\cup((-\tau_1)\times\{1\})\cup((-\varepsilon)\times\{0\}\times [0,1])\,,
\]
as illustrated in Figure~\ref{fig:cobordism}. Note that a~$(\emptyset,\emptyset)$-cobordism is nothing but a closed oriented surface embedded in the~$4$-ball.
Two~$(\tau_1,\tau_2)$-cobordisms~$\Sigma$ and~$\Sigma'$ are {\em isotopic\/} if there exists an isotopy~$h_t$ of~$D^2\times [0,1] \times [0,1]$,
keeping~$\partial(D^2\times[0,1]\times[0,1])$ fixed, such that~$h_1\vert_{\Sigma}\colon \Sigma\simeq \Sigma'$ is an orientation-preserving homeomorphism
and~$h_t(\Sigma)$ is a~$(\tau_1,\tau_2)$-cobordism for all~$t$. We shall denote by~$\Sigma\colon\tau_1\Rightarrow\tau_2$ the isotopy class of
a~$(\tau_1,\tau_2)$-cobordism~$\Sigma$, and by~$\mathit{Id}_\tau$ the isotopy class of the trivial~$(\tau,\tau)$-cobordism~$(D^2\times [0,1],\tau)\times [0,1]$. 

Fix a~$(\tau_1,\tau_2)$-cobordism~$\Sigma$ and a~$(\tau_2,\tau_3)$-cobordism~$\Sigma'$. Their {\em vertical composition\/} is the~$(\tau_1,\tau_3)$-cobordism~$\Sigma_2\star\Sigma_1$
obtained by gluing the two~$4$-balls along the cylinders containing~$\tau_2$, and shrinking the height of the resulting~$4$-ball~$D^2\times[0,1]\times[0,2]$ by a factor~$2$ (see Figure~\ref{fig:vertical}).
Finally, fix~$(\varepsilon,\varepsilon')$-tangles~$\tau_1,\tau_2$ and~$(\varepsilon',\varepsilon'')$-tangles~$\tau_3,\tau_4$.
Given a~$(\tau_1,\tau_2)$-cobordism~$\Sigma_1$ and a~$(\tau_3,\tau_4)$-cobordism~$\Sigma_2$, their {\em horizontal composition\/} is the~$(\tau_3\circ\tau_1,\tau_4\circ\tau_2)$-cobordism~$\Sigma_2\bullet\Sigma_1$ obtained by gluing the two~$4$-balls along the cylinder~$D^2 \times [0,1]$ corresponding to~$\varepsilon'$, and shrinking the length of the resulting~$4$-ball by a factor~$2$ (Figure~\ref{fig:horizontal}).

\begin{figure}[tb]
\labellist\small\hair 2.5pt
\pinlabel {$\tau_1$} at 320 1275
\pinlabel {$\tau_2$} at 375 730
\pinlabel {$\tau_2$} at 500 585
\pinlabel {$\tau_3$} at 275 -30
\pinlabel {$\Sigma_1$} at 300 1010
\pinlabel {$\Sigma_2$} at 335 360
\pinlabel {$\star$} at 295 655
\pinlabel {$=$} at 740 760
\pinlabel {$\tau_1$} at 1260 955
\pinlabel {$\tau_3$} at 1150 345
\pinlabel {$\Sigma_2\star\Sigma_1$} at 1230 700
\endlabellist
\centering
\includegraphics[width=0.5\textwidth]{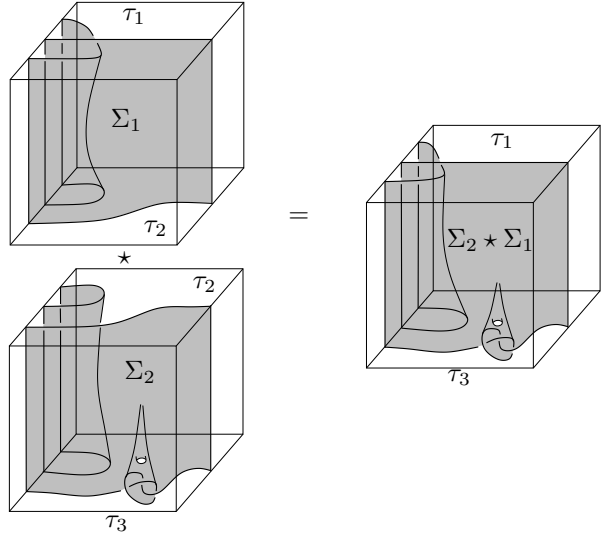}
\caption{The vertical composition of a~$(\tau_1,\tau_2)$-cobordism~$\Sigma_1$ and a~$(\tau_2,\tau_3)$-cobordism~$\Sigma_2$, the~$(\tau_1,\tau_3)$-cobordism~$\Sigma_2\star\Sigma_1$.}
\label{fig:vertical}
\end{figure}

\begin{figure}[tb]
\labellist\small\hair 2.5pt
\pinlabel {$\tau_1$} at 335 580
\pinlabel {$\tau_2$} at 353 45
\pinlabel {$\tau_3$} at 935 610
\pinlabel {$\tau_4$} at 825 45
\pinlabel {$\Sigma_1$} at 312 336
\pinlabel {$\Sigma_2$} at 900 275
\pinlabel {$\bullet$} at 625 300
\pinlabel {$=$} at 1380 300
\pinlabel {$\Sigma_2\bullet\Sigma_1$} at 2250 300
\pinlabel {$\tau_3\circ\tau_1$} at 1845 660
\pinlabel {$\tau_4\circ\tau_2$} at 1760 45
\endlabellist
\centering
\includegraphics[width=0.7\textwidth]{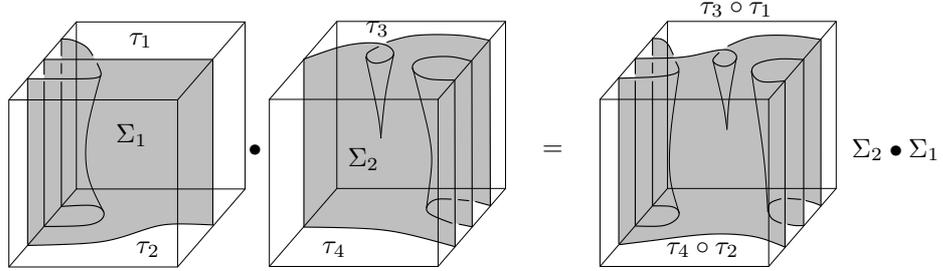}
\caption{The horizontal composition of a~$(\tau_1,\tau_2)$-cobordism~$\Sigma_1$ and a~$(\tau_3,\tau_4)$-cobordism~$\Sigma_2$,
the~$(\tau_3\circ\tau_1,\tau_4\circ\tau_2)$-cobordism~$\Sigma_2\bullet\Sigma_1$.}
\label{fig:horizontal}
\end{figure}

The {\em bicategory of oriented tangles\/} can now be defined as follows: the objects are the finite sequences of~$\pm 1$'s,
the~$1$-morphisms are given by the tangles, and the~$2$-morphisms are given by isotopy classes of cobordisms as described above. Finally, 
the associativity and identity isomorphisms
\[
a\colon(\tau_3\circ\tau_2)\circ\tau_1\Rightarrow\tau_3\circ(\tau_2\circ\tau_1),\quad \ell_{\tau}\colon I_{\varepsilon'}\circ\tau\Rightarrow\tau,
\quad r_\tau\colon\tau\circ I_{\varepsilon}\Rightarrow\tau
\]
are given by the trace of the obvious isotopies. It is a routine check to verify that all the axioms of a bicategory are satisfied.

%%%%%

\subsection{The weak~$2$-functor}
\label{sub:2mor}

We are now ready to define our weak~$2$-functor from the bicategory of oriented tangles to the bicategory of Lagrangian cospans.
Recall from subsection~\ref{sub:2-cat} that we must associate a Hermitian~$\Lambda$-module~$\mathcal{B}(\varepsilon)$ to each object~$\varepsilon$,
a cospan~$\mathcal{B}(\tau)$ to each tangle~$\tau$ and an isomorphism class of~$2$-cospans to each cobordism~$\Sigma$. Additionally, for each~$\varepsilon$, we must define an
invertible~$2$-morphism~$\varphi_\varepsilon\colon I_{\mathcal{B}(\varepsilon)} \Rightarrow \mathcal{B}(I_\varepsilon)$ and for each pair~$\tau_1, \tau_2$ of
composable tangles,
an invertible~$2$-morphism~$\varphi_{\tau_1\tau_2} \colon \mathcal{B}(\tau_2) \circ \mathcal{B}(\tau_1) \Rightarrow \mathcal{B}(\tau_2\circ\tau_1)$.

Let us associate to each object~$\varepsilon$ the Hermitian~$\Lambda$-module~$\mathcal{B}(\varepsilon)=(H_1(\widehat{D}_\varepsilon),\omega_\varepsilon)$
and to each~$(\varepsilon,\varepsilon')$-tangle~$\tau$ the Lagrangian cospan~$\mathcal{B}(\tau)$ given
by~$H_1(\widehat{D}_\varepsilon) \stackrel{i_\tau}\longrightarrow H_1(\widehat{X}_\tau) \stackrel{i_\tau'}\longleftarrow H_1(\widehat{D}_{\varepsilon'})$.
(Note that we slightly abuse notations here, as~$\mathcal{B}(\tau)$ now no longer stands for the isomorphism class of this cospan, but for the cospan itself.)
As for~$2$-morphisms, we proceed as follows.

Fix two~$(\varepsilon,\varepsilon')$-tangles~$\tau_1,\tau_2$. Given a~$(\tau_1,\tau_2)$-cobordism~$\Sigma \subset D^2 \times [0,1] \times [0,1]$, denote by~$\mathcal{N}(\Sigma)$ an open tubular neighborhood of~$\Sigma$ in~$D^2 \times [0,1] \times [0,1]$. We shall orient the exterior
\[
W_\Sigma = \left\{\begin{array}{lr}
        (D^2 \times [0,1] \times [0,1]) \setminus \mathcal{N}(\Sigma) & \text{if } \ell_\varepsilon\neq 0\\
        (S^2 \times [0,1] \times [0,1]) \setminus \mathcal{N}(\Sigma) & \text{if } \ell_\varepsilon=0
        \end{array}\right.
\]
of~$\Sigma$ so that the induced orientation on~$\partial W_\Sigma$ extends the orientation on the space~$(-X_{\tau_1}) \sqcup X_{\tau_2}$. Clearly,~$H_1(W_\Sigma)$ is generated by the
(oriented) meridians of the connected components of~$\Sigma$. The homomorphism~$H_1(W_\Sigma)\rightarrow \mathbb{Z}$ obtained by mapping these meridians to~$1$ extends the previously defined homomorphisms~$H_1(X_{\tau_1}) \rightarrow \mathbb{Z}$ and~$H_1(X_{\tau_2})\rightarrow \mathbb{Z}$. It determines an infinite cyclic covering~$\widehat{W}_\Sigma \rightarrow W_\Sigma$ whose homology is  endowed with a structure of module over~$\Lambda$.

Denote by~$H_1(\widehat{D}_\varepsilon) \stackrel{i_1}\longrightarrow H_1(\widehat{X}_{\tau_1}) \stackrel{i_1'}\longleftarrow H_1(\widehat{D}_{\varepsilon'})$ and~$H_1(\widehat{D}_\varepsilon) \stackrel{i_2}\longrightarrow H_1(\widehat{X}_{\tau}) \stackrel{i_2'}\longleftarrow H_1(\widehat{D}_{\varepsilon'})$ the Lagrangian cospans arising from~$\tau_1$ and~$\tau_2$, and let~$\alpha_{1}\colon H_1(\widehat{X}_{\tau_1}) \rightarrow H_1(\widehat{W}_\Sigma)$ and~$\alpha_{2}\colon H_1(\widehat{X}_{\tau_2}) \rightarrow H_1(\widehat{W}_\Sigma)$ be the homomorphisms induced by the inclusions of~$\widehat{X}_{\tau_1}$ and~$\widehat{X}_{\tau_2}$ into~$\widehat{W}_\Sigma$.  Combining all these inclusion induced maps, the following diagram commutes 
\[
\xymatrix@R0.5cm{ 
&  H_1(\widehat{X}_{\tau_1})  \ar[d]^{\alpha_1}  & \\
 H_1(\widehat{D}_{\varepsilon}) \ar[ur]^{i_1} \ar[dr]_{i_2} & H_1(\widehat{W}_\Sigma) & H_1(\widehat{D}_{\varepsilon'})\ar[ul]_{i_1'}\ar[dl]^{i_2'}\,. \\
&  H_1(\widehat{X}_{\tau_2})  \ar[u]_{\alpha_2}  & 
}
\]
Hence,~$H_1(\widehat{X}_{\tau_1}) \stackrel{\alpha_1}\longrightarrow H_1(\widehat{W}_\Sigma) \stackrel{\alpha_2}\longleftarrow H_1(\widehat{X}_{\tau_2})$ is a 2-cospan,
whose isomorphism class we denote by~$\mathcal{B}(\Sigma)\colon\mathcal{B}(\tau_1)\Rightarrow\mathcal{B}(\tau_2)$. 

Given any object~$\varepsilon$, let~$\alpha_\varepsilon\colon H_1(\widehat{D}_{\varepsilon})\to H_1(\widehat{X}_{I_\varepsilon})$ denote the isomorphism of~$\Lambda$-modules
induced by the inclusion of~$D_\varepsilon$ in~$D_\varepsilon\times[0,1]=X_{I_\varepsilon}$. This isomorphism fits in the commutative diagram
\[
\xymatrix@R0.5cm{ 
&  H_1(\widehat{D}_{\varepsilon})  \ar[d]^{\alpha_\varepsilon}  & \\
 H_1(\widehat{D}_{\varepsilon}) \ar[ur]^{\id} \ar[dr]_{\alpha_\varepsilon} & H_1(\widehat{X}_{I_\varepsilon}) & H_1(\widehat{D}_{\varepsilon})\ar[ul]_\id\ar[dl]^{\alpha_\varepsilon} \,. \\
&  H_1(\widehat{X}_{I_\varepsilon})  \ar[u]_\id & 
}
\]
By Remark~\ref{rem:vert}, the~$2$-morphism~$\varphi_\varepsilon\colon I_{\mathcal{B}(\varepsilon)}\Rightarrow\mathcal{B}(I_\varepsilon)$ defined by this diagram is
invertible, as required in the definition of a weak~$2$-functor.

Given an~$(\varepsilon,\varepsilon')$-tangle~$\tau_1$ and an~$(\varepsilon',\varepsilon'')$-tangle~$\tau_2$, the first part of the proof of Theorem~\ref{thm:functor} actually
shows that there is a canonical isomorphism~$\alpha_{\tau_1\tau_2}\colon H_1(\widehat{X}_{\tau_2})\circ H_1(\widehat{X}_{\tau_1})\to H_1(\widehat{X}_{\tau_2\circ\tau_1})$ which
fits in the commutative diagram
\begin{equation}
\label{eq:DefPhiTau}
\xymatrix@R0.5cm{ 
& & H_1(\widehat{X}_{\tau_2})\circ H_1(\widehat{X}_{\tau_1}) \ar[d]^{\alpha_{\tau_1\tau_2}} & & \\
H_1(\widehat{D}_{\varepsilon}) \ar@/^1pc/[urr]^{j_1i_1} \ar@/_1pc/[drr]_{k_1i_1} \ar[r]^{i_1} & H_1(\widehat{X}_{\tau_1}) \ar[ur]_{j_1} \ar[dr]^{k_1} &
H_1(\widehat{X}_{\tau_2\circ\tau_1}) &
H_1(\widehat{X}_{\tau_2}) \ar[ul]^{j_2} \ar[dl]_{k_2} & H_1(\widehat{D}_{\varepsilon''})\ar@/_1pc/[ull]_{j_2i_2''}\ar@/^1pc/[dll]^{k_2i_2''}\ar[l]_{i_2''}\,, \\
& & H_1(\widehat{X}_{\tau_2\circ\tau_1})  \ar[u]_\id & &
}
\end{equation}
where we follow the notations of the aforementioned proof. Hence, this defines a
canonical~$2$-morphism~$\varphi_{\tau_1\tau_2}\colon\mathcal{B}(\tau_1)\circ\mathcal{B}(\tau_2)\Rightarrow\mathcal{B}(\tau_2\circ\tau_1)$, which is invertible by Remark~\ref{rem:vert}.

\begin{theorem}
\label{thm:2functor}
$\mathcal{B}$ together with the isomorphisms~$\varphi_\varepsilon$ and~$\varphi_{\tau_1\tau_2}$ gives rise to a weak~$2$-functor from the bicategory of oriented tangles to the bicategory of Lagrangian cospans, whose restriction to oriented surfaces is given by the Alexander module.
\end{theorem}

\begin{proof}
First note that isotopic cobordisms define isomorphic~$2$-cospans, so~$\mathcal{B}$ is well-defined at the level of~$2$-morphisms. Also, for any tangle~$\tau$,~$\mathcal{B}$
clearly maps the trivial concordance~$\mathit{Id}_\tau$ to a~$2$-cospan canonically isomorphic to the identity~$2$-cospan~$\mathit{Id}_{\mathcal{B}(\tau)}$.

Let us now verify that~$\mathcal{B}$ preserves the vertical composition. Fix a~$(\tau_1,\tau_2)$-cobordism~$A$ and a~$(\tau_2,\tau_3)$-cobordism~$B$. 
Let~$H_1(\widehat{X}_{\tau_1}) \stackrel{\alpha_1}\rightarrow H_1(\widehat{W}_A) \stackrel{\alpha_2}\leftarrow H_1(\widehat{X}_{\tau_2})$ and~$H_1(\widehat{X}_{\tau_2}) \stackrel{\beta_2}\rightarrow H_1(\widehat{W}_B) \stackrel{\beta_3}\leftarrow H_1(\widehat{X}_{\tau_3})$ be the 2-cospans arising from~$A$ and~$B$.  We need to show
that~$H_1(\widehat{X}_{\tau_1}) \stackrel{k_A\alpha_1}{\longrightarrow} H_1(\widehat{W}_{B \star A}) \stackrel{k_B\beta_3}\longleftarrow H_1(\widehat{X}_{\tau_3})$ is isomorphic
to the vertical composition~$H_1(\widehat{X}_{\tau_1}) \stackrel{v_A\alpha_1}{\longrightarrow} H_1(\widehat{W}_B) \star H_1(\widehat{W}_A) \stackrel{v_B\beta_3}{\longleftarrow} H_1(\widehat{X}_{\tau_3})$, where~$k_A,k_B$ are the inclusion induced maps and~$v_A,v_B$ are maps resulting from any representative of the
pushout~$H_1(\widehat{W}_B) \star H_1(\widehat{W}_A)$. Observe that~$\widehat{W}_{B \star A}$ decomposes as the union of~$\widehat{W}_B$ and~$\widehat{W}_A$ glued along~$\widehat{X}_{\tau_2}$. Therefore, the associated Mayer-Vietoris exact sequence
\[
\xymatrix@R0.5cm{
H_1(\widehat{X}_{\tau_2})\ar[r]^{\!\!\!\!\!\!\!\!\!\!\!\!\!\!\!\!\!\!\!(-\alpha_2,\beta_2)}&H_1(\widehat{W}_A)\oplus H_1(\widehat{W}_B)\ar[r]^{\;\;\;\;\;\;\;\;\;\;{k_A \choose k_B}}&H_1(\widehat{W}_{B \star A})\ar[r]& 0
}
\]
together with Lemma~\ref{lemma:pushout} imply that~$H_1(\widehat{W}_A) \stackrel{k_A}{\longrightarrow} H_1(\widehat{W}_{B \star A}) \stackrel{k_B}\longleftarrow H_1(\widehat{W}_B)$ is a representative for the pushout~$H_1(\widehat{W}_A) \star H_1(\widehat{W}_B)$. Consequently, these two cospans are canonically isomorphic and the claim follows.

Given tangles and cobordisms as illustrated below
\[
\xymatrix@C+2pc{\varepsilon\rtwocell^{\tau_1}_{\tau_2}{\;\;\;A}&\varepsilon'\rtwocell^{\tau_3}_{\tau_4}{\;\;\;B}&\varepsilon''\,,}
\]
our next goal is to prove the equality
\begin{equation}
\label{eq:naturality}
\varphi_{\tau_2\tau_4}\star(\mathcal{B}(B)\bullet\mathcal{B}(A))=\mathcal{B}(B\bullet A)\star\varphi_{\tau_1\tau_3}
\end{equation}
up to isomorphism of~$2$-cospans. Since the~$2$-morphism~$\varphi_{\tau_1\tau_3}$ is represented by the~$2$-cospan~$H_1(\widehat{X}_{\tau_3}) \circ H_1(\widehat{X}_{\tau_1})\stackrel{\alpha_{\tau_1\tau_3}}{\longrightarrow} H_1(\widehat{X}_{\tau_3\circ\tau_1}) \stackrel{\id}{\longleftarrow}H_1(\widehat{X}_{\tau_3\circ\tau_1})$, Remark~\ref{rem:vert} implies that
the right hand side of equation (\ref{eq:naturality}) is represented by the~$2$-cospan
\[
H_1(\widehat{X}_{\tau_3}) \circ H_1(\widehat{X}_{\tau_1}) \stackrel{k_{31}\alpha_{\tau_1\tau_3}}{\longrightarrow} H_1(\widehat{W}_{B \bullet A}) \stackrel{k_{42}}\longleftarrow H_1(\widehat{X}_{\tau_4\circ\tau_2})\,,
\]
where~$k_{31}$ and~$k_{42}$ are induced by the inclusion maps. A similar argument shows that the left hand side of equation (\ref{eq:naturality}) is represented by the~$2$-cospan
\[
H_1(\widehat{X}_{\tau_3}) \circ H_1(\widehat{X}_{\tau_1}) \stackrel{h_{31}}{\longrightarrow} H_1(\widehat{W}_B) \bullet H_1(\widehat{W}_A) \stackrel{h_{42}\alpha^{-1}_{\tau_2\tau_4}}{\longleftarrow} H_1(\widehat{X}_{\tau_4\circ\tau_2})\,,
\]
where this time, the maps~$h_{31}$ and~$h_{42}$ are the ones which arise from the definition of horizontal composition. It now remains to find an isomorphism~$f$ of~$\Lambda$-modules which fits in the following commutative diagram:
\[
\xymatrix@R0.5cm{ 
& H_1(\widehat{X}_{\tau_3}) \circ H_1(\widehat{X}_{\tau_1})  \ar[dl]_{h_{31}} \ar[dr]^{k_{31}\alpha_{\tau_1\tau_3}} & \\
H_1(\widehat{W}_B) \bullet H_1(\widehat{W}_A) \ar[rr]^f & & \;H_1(\widehat{W}_{B\bullet A})\,.  \\
& H_1(\widehat{X}_{\tau_4\circ\tau_2}) \ar[lu]^{h_{42}\alpha_{\tau_2\tau_4}^{-1}}\ar[ur]_{k_{42}}  & 
}
\]
In order to construct~$f$, first observe that the following diagram commutes
\[
\xymatrix@R0.5cm{
H_1(\widehat{X}_{\tau_2})\ar[d]_{\alpha_2} & H_1(\widehat{D}_{\varepsilon'}) \ar[d]^{\cong}\ar[r]^{i_3'}\ar[l]_{i_2'}&  H_1(\widehat{X}_{\tau_3})\ar[d]^{\beta_3}   \\
H_1(\widehat{W}_A) & H_1(\widehat{D}_{\varepsilon'} \times [0,1])\ar[l]\ar[r]   & H_1(\widehat{W}_B),
}
\]
where all the maps are induced by inclusions. Hence, identifying~$H_1(\widehat{D}_{\varepsilon'}\times [0,1])$ with~$H_1(\widehat{D}_{\varepsilon'})$, the first map in the Mayer-Vietoris exact sequence
\[
\xymatrix@R0.5cm{
H_1(\widehat{D}_{\varepsilon'})\ar[r]^{\!\!\!\!\!\!\!\!\!\!\!\!\!\!\!\!\!\!\!}&H_1(\widehat{W}_A)\oplus H_1(\widehat{W}_B)\ar[r]^{\;\;\;\;\;\;\;\;\;\;}&H_1(\widehat{W}_{B \bullet A})\ar[r]& 0
}
\]
is given by~$(-\alpha_2i_2',\beta_3i_3')$.  
It now follows from Lemma~\ref{lemma:pushout} that the cospan of inclusion induced maps~$H_1(\widehat{W}_A) \stackrel{k_A}{\to} H_1(\widehat{W}_{B \bullet A}) \stackrel{k_B}{\leftarrow} H_1(\widehat{W}_B)$ is a representative of the pushout~$H_1(\widehat{W}_A) \stackrel{h_A}{\to}H_1(\widehat{W}_B) \bullet H_1(\widehat{W}_A)\stackrel{h_B}{\leftarrow} H_1(\widehat{W}_B)$. Invoking the corresponding universal property, this produces a~$\Lambda$-module isomorphism~$f \colon H_1(\widehat{W}_B) \bullet H_1(\widehat{W}_A) \to H_1(\widehat{W}_{B\bullet A})$ with~$fh_A=k_A$ and~$fh_B=k_B$. Using successively the definition of~$\alpha_{\tau_1\tau_3}$ (see diagram (\ref{eq:DefPhiTau}) for the relevant notations), the commutativity of inclusion induced maps, and the equalities above, one gets
\[
f^{-1}k_{31} \alpha_{\tau_1\tau_3}j_3=f^{-1}k_{31}k_3=f^{-1}k_B\beta_3=h_B\beta_3\,.
\]
The equality~$f^{-1}k_{31} \alpha_{\tau_1\tau_3}j_1=h_A\alpha_1$ is proved similarly. Hence, the universal property of diagram~(\ref{equ:LegsForHori}) implies
that~$h_{31}=f^{-1}k_{31} \alpha_{\tau_1\tau_3}$. The equality~$h_{42}=f^{-1}k_{42} \alpha_{\tau_2\tau_4}$ can be dealt with in the same way, and equation (\ref{eq:naturality}) is proved.

Given an~$(\varepsilon,\varepsilon')$-tangle~$\tau$, we must now show that the~$2$-morphism~$\mathcal{B}(r_\tau) \star \varphi_{I_{\varepsilon}\tau} \star (I_{\mathcal{B}(\tau)} \bullet \varphi_\varepsilon)$ coincides with~$r_{\mathcal{B}(\tau)} \colon \mathcal{B}(\tau) \circ I_{\mathcal{B}(\varepsilon)} \Rightarrow \mathcal{B}(\tau)$. First observe that by Remark~\ref{rem:cospan}, one can choose representatives of the pushouts so that for any cospan~$H \to T \leftarrow H'$, one has~$T \circ I_H=T$. In particular, we only need to prove that, for this choice of pushouts,
\begin{equation}
\label{eq:CoherenceId}
\mathcal{B}(r_\tau)  \star \varphi_{I_{\varepsilon}\tau}  \star (I_{\mathcal{B}(\tau)} \bullet \varphi_\varepsilon) =I_{\mathcal{B}(\tau)}.
\end{equation}
As a first step, using the definition of the horizontal composition and Remark~\ref{rem:cospan}, we deduce that~$I_{\mathcal{B}(\tau)} \bullet \varphi_\varepsilon$ is represented by the~$2$-cospan
\[
H_1(\widehat{X}_\tau) \stackrel{\id}{\longrightarrow} H_1(\widehat{X}_\tau) \stackrel{h}{\longleftarrow} H_1(\widehat{X}_\tau) \circ H_1(\widehat{X}_{I_\varepsilon})\,,
\]
where~$h$ is the unique morphism which fits in the following commutative diagram (recall diagram~(\ref{equ:LegsForHori})):
\[
\xymatrix@R0.5cm{
& & H_1(\widehat{X}_\tau) & & & \\
& & H_1(\widehat{X}_\tau) \circ H_1(\widehat{X}_{I_\varepsilon}) \ar[u]^{h} & & & \\
& H_1(\widehat{D}_\varepsilon) \ar@/^1pc/[uur]^{i} &  & H_1(\widehat{X}_\tau) \ar@/_1pc/[uul]_{\id} &  & \\
& H_1(\widehat{X}_{I_\varepsilon}) \ar[ruu]^{j_1} \ar[u]^{\alpha_\varepsilon^{-1}} & & H_1(\widehat{X}_\tau). \ar[uul]_{j_2} \ar[u]_{\id} \\
 & & H_1(\widehat{D}_\varepsilon) \ar[lu]_{\alpha_\varepsilon}\ar[ru]^{i} & &}
\]
A short computation using Remark~\ref{rem:cospan} then shows that the left hand side of equation~(\ref{eq:CoherenceId}) is represented by the~$2$-cospan
\[
H_1(\widehat{X}_\tau) \stackrel{\alpha_{I_\varepsilon \tau} h^{-1}}{\longrightarrow} H_1(\widehat{X}_{\tau\circ I_\varepsilon}) \stackrel{r^{-1}}{\longleftarrow} H_1(\widehat{X}_\tau)\,,
\]
where~$r\colon H_1(\widehat{X}_{\tau\circ I_\varepsilon})\to H_1(\widehat{X}_\tau)$ is the isomorphism induced by the obvious isotopy from~$\tau\circ I_\varepsilon$ to~$\tau$.
We now claim that~$r$ induces a~$2$-cospan isomorphism from~$I_{\mathcal{B}(\tau)}$ to this cospan. To prove this claim, we only need to show the
equality~$r\alpha_{I_\varepsilon \tau} h^{-1}=\id_{H_1(\widehat{X}_{\tau})}$, i.e. to check that~$r\alpha_{I_\varepsilon \tau}$ satisfies the defining property of~$h$ displayed above. Since~$\alpha_{I_\varepsilon \tau}j_1$ and~$\alpha_{I_\varepsilon \tau}j_2$ are the inclusion induced homomorphisms (recall diagram (\ref{eq:DefPhiTau})), this follows
from the functoriality of homology. The proof of the equality~$\mathcal{B}(\ell_\tau)  \star \varphi_{\tau I_{\varepsilon'}} \star (\varphi_{\varepsilon'} \bullet I_{\mathcal{B}(\tau)})=\ell_{\mathcal{B}(\tau)}$ is dealt with in the same way.

Finally, the axiom involving the associativity isomorphisms is left to the reader: although the proof is tedious, it involves no other ideas than the ones presented up to now. Therefore we have proved that~$\mathcal{B}$ is a weak~$2$-functor and we turn to the last statement of the theorem. If~$\Sigma$ is a~$(\emptyset,\emptyset)$-cobordism, that is, a closed oriented surface in the~$4$-ball, then the associated~$2$-cospan is given by
\[
\xymatrix@R0.5cm{ 
&  0  \ar[d]  & \\
0 \ar[ur] \ar[dr] & H_1(\widehat{W}_\Sigma) & 0 \ar[ul] \ar[dl]\,, \\
&  0  \ar[u]  & 
}
\]
with~$W_\Sigma$ the complement of~$\Sigma$ in the~$4$-ball. This is nothing but the Alexander module of~$\Sigma$.
\end{proof}

%%%%%%%%%%%%%%%%%%%%%%%%

\section{Unreduced and multivariable versions}
\label{sec:versions}

Recall that the representation originally defined by Burau takes the form of a homomorphism~$\overline{\rho}_n\colon B_n\to\mathit{GL}_n(\Lambda)$, which is the direct sum
of a trivial~$1$-dimensional representation with~$\rho_n\colon B_n\to\mathit{GL}_{n-1}(\Lambda)$, the {\em reduced\/} Burau representation. Also, these representations admit
multivariable extensions, the so-called {\em Gassner representations\/} of the pure braid groups. It is therefore natural to ask whether these variations of the Burau representation can also be extended to weak~$2$-functors. This is indeed the case, and is the subject of this slightly informal last section.

More precisely, we start in subsection~\ref{sub:monoidal} by explaining how~$\overline{\rho}$ can be extended to a functor~$\overline{\mathcal{B}}$ on tangles.
This functor is no longer Lagrangian ($\overline{\rho}$ is not unitary) but it is monoidal and behaves well with respect to traces. In subsection~\ref{sub:2-monoidal},
we indicate how to extend it to a weak~$2$-functor. Finally, in subsection~\ref{sub:multi}, we briefly explain how all of these constructions can be extended
to multivariable versions, defined on the category of colored tangles.

%%%%%

\subsection{Extending the unreduced Burau representation to a monoidal functor}
\label{sub:monoidal}

Given an integral domain~$\Lambda$, let~$\mathbf{C}_\Lambda$ denote the category
with finitely generated~$\Lambda$-modules as objects, and isomorphism classes of cospans as morphisms, composed by pushouts. Also, let~$\mathbf{GL}_\Lambda$ denote the groupoid
with the same objects as~$\mathbf{C}_\Lambda$ and~$\Lambda$-isomorphisms as morphisms. As in Section~\ref{sec:Lagr}, one can check that the map assigning to an invertible
cospan~$H\stackrel{i}{\longrightarrow}T\stackrel{i'}{\longleftarrow}H'$ the~$\Lambda$-isomorphism~$i'^{-1}i\colon H\to H'$ defines an equivalence of
categories~$\mathit{core}(\mathbf{C}_\Lambda) \stackrel{\simeq}{\longrightarrow}\mathbf{GL}_\Lambda$.
Note that the direct sum endows these categories with a monoidal structure, with the trivial~$\Lambda$-module~$H=0$ being the identity object. Given an endomorphism
of~$\mathbf{C}_\Lambda$, i.e. a cospan of the form~$H\stackrel{i}{\longrightarrow}T\stackrel{i'}{\longleftarrow}H$, define the {\em trace\/} of~$T$ as the coequalizer
\[
\xymatrix@R0.5cm{
H \ar@/^/[rr]^{i}\ar@/_/[rr]_{i'}& & T \ar[r]^{\!\!\!j}&\mathrm{tr}(T)\,.
}
\]
Viewing~$\mathrm{tr}(T)$ as the isomorphism class of the cospan~$0\rightarrow\mathrm{tr}(T)\leftarrow 0$, the trace actually defines a
map~$\mathrm{tr}\colon\mathit{End}(H) \rightarrow \mathit{End}(0)$. It is an amusing exercise to check that it satisfies the following properties, as it should
(see e.g.~\cite[p.~22]{Turaev}).
\begin{romanlist}
\item{If~$T_1$ is a cospan from~$H$ to~$H'$ and~$T_2$ from~$H'$ to~$H$, then~$\mathrm{tr}(T_1\circ T_2)=\mathrm{tr}(T_2 \circ T_1)$.}
\item{If~$T_1$ and~$T_2$ are two endomorphisms, then~$\mathrm{tr}(T_1\oplus T_2)=\mathrm{tr}(T_1)\circ\mathrm{tr}(T_2)$.}
\item{If~$T$ is an endomorphism of~$0$, then~$\mathrm{tr}(T)=T$.}
\end{romanlist}

These additional structures are also present in the category of tangles. Indeed, the juxtaposition endows~$\mathbf{Tangles}$ with a monoidal structure,
with the empty set~$\varepsilon=\emptyset$ being the identity object. Furthermore, the closure of a tangle defines a natural trace
function~$\mathit{End}(\varepsilon)\to\mathit{End}(\emptyset)$. In this context, the unreduced Burau representation can be understood as a monoidal functor~$\overline{\rho}\colon\mathbf{Braids}\to\mathbf{GL}_\Lambda$, where~$\Lambda=\Z[t^{\pm 1}]$.

We now sketch the construction of a monoidal functor~$\overline{\mathcal{B}}\colon\mathbf{Tangles}\to\mathbf{C}_\Lambda$ extending~$\overline{\rho}$, and behaving well with respect
to traces. We shall follow the notation of Section~\ref{sec:2functor}, apart from the fact that all exteriors will be considered in the unit disc~$D^2$, and not the sphere~$S^2$
even when~$\ell_\varepsilon$ vanishes. Let~$x_0$ be the point~$(-1,0)$ in~$D^2$.
For any sequence~$\varepsilon$ of~$\pm 1$'s, set~$\overline{\mathcal{B}}(\varepsilon)=H_1(\widehat{D}_\varepsilon,\widehat{x_0})$ and for any isotopy class~$\tau$ of tangles, let~$\overline{\mathcal{B}}(\tau)$ denote the isomorphism class of the cospan~$H_1(\widehat{D}_\varepsilon,\widehat{x_0}) \stackrel{i_\tau}\longrightarrow H_1(\widehat{X}_\tau,\widehat{x_0\times I})\stackrel{i_\tau'}\longleftarrow H_1(\widehat{D}_{\varepsilon'},\widehat{x_0})$, where~$\widehat{Y}$ stands for the inverse image of a subspace~$Y\subset X_\tau$
by the infinite cyclic covering map~$\widehat{X_\tau}\to X_\tau$. Following almost {\em verbatim\/} the proof of Theorem~\ref{thm:functor}, one checks that this defines
a functor~$\overline{\mathcal{B}}\colon\mathbf{Tangles}\to\mathbf{C}_\Lambda$ which fits in the commutative diagram 
\[
\xymatrix@R0.5cm{ 
& \mathbf{Braids} \ar[d] \ar[r] \ar@/_1pc/[ld]_{\overline{\rho}} & \mathbf{Tangles} \ar[d]^{\overline{\mathcal{B}}} \\
\mathbf{GL}_\Lambda &\mathit{core}(\mathbf{C}_\Lambda) \ar[l]_\simeq\ar[r] & \mathbf{C}_\Lambda\,.}
\]
Furthermore, an additional application of Mayer-Vietoris shows that this functor is monoidal. (The basepoint~$x_0$ is chosen so that the juxtaposition of tangles can be realized
in a natural way by gluing discs along intervals, with~$x_0$ a common endpoint of these intervals.)
Finally, if~$\tau$ is an~$(\varepsilon,\varepsilon)$-tangle, then~$\mathrm{tr}(\overline{\mathcal{B}}(\tau))$ is nothing but the relative Alexander module of the oriented link
in~$D^2\times I$ (or equivalently, in~$S^3$) obtained by the closure of~$\tau$.

%%%%%

\subsection{$\overline{\mathcal{B}}$ as a monoidal weak~$2$-functor}
\label{sub:2-monoidal}

One can modify~$\mathbf{C}_\Lambda$ to obtain a bicategory in the exact same way as we did for~$\mathbf{L}_\Lambda$,
with~$2$-morphisms given by isomorphism classes of~$2$-cospans (recall subsection~\ref{sub:2-cospan}). Furthermore, the direct sum endows this bicategory
with a monoidal structure.

Also, the juxtaposition endows the bicategory of tangles with a monoidal structure. Here again, some care is needed, as different conventions such as the ones in~\cite{KT}
and~\cite{BL1} will lead to different monoidal~bicategories. We will not go into these details, but only mention that our construction is robust enough to be valid
in these different settings.

Let us sketch how the functor~$\overline{\mathcal{B}}$ can be extended to a weak~$2$-functor, following the notation of subsection~\ref{sub:2mor}.
Given a~$(\tau_1,\tau_2)$-cobordism~$\Sigma$, let us denote by~$\overline{\mathcal{B}}(\Sigma)\colon\overline{\mathcal{B}}(\tau_1)\Rightarrow\overline{\mathcal{B}}(\tau_2)$ the
isomorphism class of the~$2$-cospan
\[
H_1(\widehat{X}_{\tau_1},\widehat{x_0\times I}) \stackrel{\alpha_1}\longrightarrow H_1(\widehat{W}_\Sigma,\widehat{x_0\times I\times I}) \stackrel{\alpha_2}\longleftarrow H_1(\widehat{X}_{\tau_2},\widehat{x_0\times I})\,.
\]
One can check that this defines a weak~$2$-functor, that is monoidal in a sense that, once again, we shall not discuss in detail here.

%%%%%

\subsection{Multivariable versions}
\label{sub:multi}

Let~$\mu$ be a positive integer. Recall that a~{\em~$\mu$-colored tangle\/} consists in an oriented tangle~$\tau$ together with a surjective map assigning to each
component of~$\tau$ an integer in~$\{1,\dots,\mu\}$. As explained in~\cite[Section~6.1]{CT},~$\mu$-colored tangles naturally form a category~$\mathbf{Tangles}_\mu$, with the~$\mu=1$ case being nothing but~$\mathbf{Tangles}$. Obviously, assigning a color to the cobordisms and proceeding as in subsection~\ref{sub:2mor},
one obtains a bicategory of~$\mu$-colored tangles.

All the results of the present paper extend to this multivariable setting in a straightforward way, that we now very briefly summarize. The coloring of
points, tangles and cobordisms induces homomorphisms from the homology of the corresponding exterior onto~$\Z^\mu$, thus defining free abelian covers
whose homology is a module over the ring of multivariable Laurent polynomials~$\Z[\Z^\mu]=\Z[t_1^{\pm 1},\dots,t_\mu^{\pm 1}]=:\Lambda_\mu$. This allows one to construct
a weak~$2$-functor from the bicategory of~$\mu$-colored tangles to the bicategory of Lagrangian cospans over~$\Lambda_\mu$, which extends the colored
Gassner representation of~$\mu$-colored braids, and whose restriction to~$\mu$-colored links and surfaces is nothing but the multivariable Alexander module.
The results of subsections~\ref{sub:monoidal} and~\ref{sub:2-monoidal} can be extended in the same way.

\bibliographystyle{plain}
\nocite{*}
\bibliography{bibliographie}

\end{document}